\documentclass[journal,twoside,web]{ieee/ieeecolor}
\let\labelindent\relax
\usepackage{enumitem}

\usepackage{ieee/generic}

\newcommand{\update}[1]{{#1}}

\usepackage{amsmath,amssymb,amsfonts}
\usepackage{booktabs}
\usepackage{cite}
\usepackage{algorithmic}
\usepackage{graphicx}
\usepackage{textcomp}
\usepackage{xcolor}
\usepackage{colortbl}
\usepackage{caption}
\usepackage{tikz}

\makeatletter
\let\NAT@parse\undefined
\makeatother
\usepackage[colorlinks]{hyperref}
\hypersetup{
  colorlinks   = true, 
  urlcolor     = blue, 
  linkcolor    = red, 
  citecolor    = green!50!black   
}

\newtheorem{innercustomgeneric}{\customgenericname}
\providecommand{\customgenericname}{}
\newcommand{\newcustomtheorem}[2]{%
  \newenvironment{#1}[1]
  {%
   \renewcommand\customgenericname{#2}%
   \renewcommand\theinnercustomgeneric{##1}%
   \innercustomgeneric
  }
  {\endinnercustomgeneric}
}

\newcustomtheorem{customthm}{Theorem}
\newcustomtheorem{customlemma}{Lemma}

\usepackage{bm, Shortcuts_OPT}
\usepackage{tcolorbox}
\newtheorem*{Theorem*}{Theorem}
\newtheorem{Def}{Definition}
\newtheorem{Corollary}{Corollary}
\newtheorem{Assumption}{\textbf{A}\!\!}

\allowdisplaybreaks[4]

\def\BibTeX{{\rm B\kern-.05em{\sc i\kern-.025em b}\kern-.08em
    T\kern-.1667em\lower.7ex\hbox{E}\kern-.125emX}}
\begin{document}
\title{Tighter Analysis for Decentralized Stochastic Gradient Method: Impact of Data Homogeneity}
\author{Qiang Li, \IEEEmembership{Student Member, IEEE}, and Hoi-To Wai, \IEEEmembership{Member, IEEE}
\thanks{Qiang Li and Hoi-To Wai are with Dept.~of System Engineering \& Engineering Management, Faculty of Engineering,
        The Chinese University of Hong Kong, Shatin District, Hong Kong SAR. Emails:~\url{liqiang@se.cuhk.edu.hk}, \url{htwai@se.cuhk.edu.hk}.}
}

\maketitle

\begin{abstract}
This paper studies the effect of data homogeneity on multi-agent stochastic optimization. We consider the decentralized stochastic gradient (DSGD) algorithm and perform a refined convergence analysis. Our analysis is explicit on the similarity between Hessian matrices of local objective functions which captures the degree of data homogeneity. We illustrate the impact of our analysis through studying the transient time, defined as the minimum number of iterations required for a distributed algorithm to achieve comparable performance as its centralized counterpart. When the local objective functions have similar Hessian, the transient time of DSGD can be as small as ${\cal O}(n^{2/3}/\rho^{8/3})$ for smooth (possibly non-convex) objective functions, ${\cal O}(\sqrt{n}/\rho)$ for strongly convex objective functions, where $n$ is the number of agents and $\rho$ is the spectral gap of graph. These findings provide a theoretical justification for the empirical success of DSGD. Our analysis relies on a novel observation with higher-order Taylor approximation for gradient maps that can be of independent interest. Numerical simulations validate our findings.
\end{abstract}


\section{Introduction}\label{sec:introduction}
\IEEEPARstart{W}{e} consider a system of $n$ agents associated via a connected and undirected graph with self loops $G = (V,E)$, where $V = [n] := \{1,\ldots,n\}$ is the set of agents, $E\subseteq N\times N$ is the edge set between the agents.
Our aim is to tackle the distributed optimization problem:
\begin{align}\label{eq:opt}
    &\min_{\prm_i\in \RR^d, i \in [n] }~  \frac{1}{n}\sum_{i=1}^{n} f_i(\prm_{i}) ~~\text{s.t.}~~\prm_i = \prm_j, ~ \forall i, j\in E,
\end{align}
where for each $i \in V$, the twice continuously differentiable local loss function $f_i:\RR^{d} \mapsto \RR$ can be accessed only by agent $i$. Every agent $i\in V$ can receive and send information only from its neighbors $\{ j: (i,j)\in E \}$. Note that as $G$ is connected, tackling \eqref{eq:opt} is equivalent to the problem of minimizing $f(\prm) := (1/n) \sum_{i=1}^n f_i( \prm )$. 
We assume that (\ref{eq:opt}) is lower bounded, i.e., $f^\star := \min_{\prm \in \RR^d} f(\prm) > -\infty$. 

The distributed optimization problem (\ref{eq:opt}) has found applications in wireless sensor networks \cite{cohen2017projected, mateos2012distributed}, multi-agent reinforcement learning \cite{doan2019finite},  distributed machine learning \cite{forrester2007multi, nedic2017fast, cohen2016distributed}, federated learning \cite{brisimi2018federated}, etc.; also see the survey papers \cite{chang2020distributed, nedic2020distributed}. We are interested in the setting where $f_i(\cdot)$ is defined on a large or streaming dataset accessed by the $i$th agent. To model this case, $f_i(\cdot)$ takes the form of a stochastic function:
\beq \label{eq:fi}
    f_i(\prm_i) \eqdef \EE_{Z_i\sim {\sf B}_i} [\ell_i (\prm_i; Z_i)],
\eeq 
such that $Z_{i}\sim {\sf B}_{i}$ represents a sample drawn from the local data distribution ${\sf B}_{i}$ and $\ell_i( \prm ; Z_i)$ is the loss at agent $i$ with the decision $\prm$ for a given sample $Z_i$.

\begin{figure*}[t]
  \centering
  \begin{minipage}{.875\textwidth}  
\begin{center}
\resizebox{\linewidth}{!}{\begin{tabular}{c c c c c c}
\toprule
\bfseries Algorithm & \bfseries Reference &\bfseries Data Distribution & \bfseries Mechanism &\bfseries $T_{\sf trans}$ [S-cvx] & \bfseries $T_{\sf trans}$ [Noncvx]
\\
\midrule
DSGD & {\bf This Work} &{\bf Homogeneous}  & Plain & ${\cal O}(\sqrt{n}/\rho)$ &  ${\cal O}(n^{5/3}/\rho^{8/3})$
\\
\midrule
DSGD & \cite{lian2017decentralized}, \cite{pu2021sharp} & Heterogeneous & Plain & ${\cal O}(n/\rho^2)$ & ${\cal O}(n^3/\rho^4)$ 
\\
Gradient Tracking & \cite{C2}, \cite{pu2021distributed}, \cite{xin2021improved} & Heterogeneous & grad.~tracking (GT) & ${\cal O}( n / \sqrt{\rho} ) $ & ${\cal O}( n^3 / \rho^2 )$ 
\\
Exact Diffusion ($D^2$) & \cite{C3}, \cite{yuan2023removing} & Heterogeneous & grad. tracking (GT) & ${\cal O}(n/\rho)$ & ${\cal O}(n^3/\rho^2)$ 
\\
EDAS & \cite{huang2022improving} & Heterogeneous & grad. tracking (GT) & ${\cal O}(n/\rho)$ & {\tt N/A}
\\
DeTAG & \cite{C4} & Heterogeneous & GT + Acc.~Gossip & {\tt N/A} & ${\cal O}(n/\rho)^\dagger$
\\
MG-DSGD & \cite{yuan2022revisiting} & Heterogeneous & Multi-Gossip & ${\cal O}(1/\sqrt{\rho})^\dagger$ & ${\cal O}(n/\rho)^\dagger$
\\
\bottomrule
\end{tabular}}
\end{center}
\end{minipage}
\captionof{table}{{Comparison of transient time analyzed in related works. {\bf [S-cvx]}, {\bf [Noncvx]} respectively denotes the strongly convex, smooth (possibly non-convex) loss case. $^\dagger$DeTAG and MG-DSGD require executing multiple rounds of (accelerated) gossiping \emph{at each iteration}.
}\vspace{-.4cm}}\label{table1}
\end{figure*}

The study of decentralized optimization algorithms for \eqref{eq:opt} can be traced back to the 1980s \cite{tsitsiklis1984problems}. Popular algorithms are based on gradient descent and gossip communication mechanism where information flows along the edges are specified by the graph. Decentralized stochastic gradient descent (\algoname) algorithm was first proposed in \cite{ram2008distributed} for tackling problem (\ref{eq:opt}) using stochastic oracles of the (sub)gradients of $f_i$ when $f_i$ is convex. Improvements to the analysis and/or the algorithms have been developed since then. For example, \cite{bianchi2012convergence} showed that {\algoname} converges to a stationary point asymptotically when $f$ is possibly non-convex. An important class of variants includes the use of gradient tracking: Nedi\'{c} et al. \cite{nedic2017achieving} combined the inexact gradient method with gradient tracking to develop the {DIGing} algorithm, Di Lorenzo et al.~\cite{di2016next} developed the {NEXT} algorithm using gradient tracking for time-varying graph and analyzed its asymptotic behavior in non-convex setting. The recent work \cite{tang2018d} proposed $D^{2}$ algorithm to reduce the large variance of stochastic gradients. In \cite{yuan2018exact}, an {Exact Diffusion} algorithm was designed under exact and stochastic gradient settings. For decentralized directed graph, \cite{assran2019stochastic} studied stochastic gradient push algorithm which blends {\algoname} and {PushSum}.

Despite its simple structure, {\algoname} is efficient in tackling practical machine learning problems. Early works such as \cite{towfic2016excess} found that {\algoname} delivers comparable  asymptotic performance as a centralized algorithm  while assuming that all local functions $f_i(\cdot)$ have the same minimum. For the general case, \cite{pu2021distributed, lian2017can} showed that {\algoname} achieves a \emph{linear speedup} in convergence rate for (strongly-)convex and non-convex objective functions. Their results show that the iterates of {\algoname} converge to an optimal solution (or stationary point) at the same \emph{asymptotic rate} as an equivalent centralized stochastic gradient ({\sf CSGD}) algorithm utilizing the same number of stochastic samples per iteration. In other words, the extra cost of applying a decentralized algorithm vanishes since the asymptotic rate is \emph{network independent}. 

An important extension to the above works is to characterize the \emph{transient time} of decentralized algorithms, i.e., the number of iterations required to achieve the aforementioned rate of {\sf CSGD} that is network independent.
Studying the \emph{transient time} has been the topic of interest recently: \cite{pu2021sharp} showed that for strongly convex loss, the transient time of {\algoname} is ${\cal O}(n/\rho^2)$, where $\rho$ is the spectral gap for the communication graph between agents; \update{also see the recent work \cite{vogels2023beyond} which improved the transient time to ${\cal O}(n/\rho)$ for a restricted class of graph topologies}. 
Other works have studied more advanced algorithms, e.g., \cite{xin2021improved} showed that the transient time of distributed stochastic gradient tracking ({\sf DSGT}) is ${\cal O}(n/\rho^3)$ {which is later on improved to ${\cal O}(n / \sqrt{\rho} )$ in \cite{C2} using a new spectral condition on the mixing matrix}. 
As $n \gg 1$, $\rho \ll 1$ in large and sparse networks, the above results indicate that the convergence rate of decentralized algorithms may still be heavily influenced by the network topology. 
It motivated recent works to develop algorithms focusing on improved transient time, e.g., \cite{yuan2022revisiting} proposed an {\sf MG-DSGD} algorithm which utilizes multiple gossiping steps to obtain an optimal convergence rate for decentralized stochastic optimization and \cite{C4} proposed a related idea, \cite{yuan2023removing, C3} considered the $D^2$/Exact-diffusion method and provide a tightened analysis through studying the algorithm from a primal-dual optimization perspective, also see \cite{huang2022improving} which uses a different proof idea. Table \ref{table1} gives a selective summary of the recent results.

Instead of seeking  further sophistication for decentralized algorithms, this paper focuses on the plain {\algoname} algorithm and develops tighter convergence theories. 
In particular, we study conditions on \eqref{eq:opt} that can be leveraged by {\algoname} to accelerate its convergence. 
We are motivated by the empirical successes of {\algoname} seen in various studies {\cite{ryabinin2021moshpit}} 
and propose a novel convergence analysis from the perspective of \emph{data homogeneity}.
Note that a common scenario is that workers/agents acquire data in an i.i.d.~fashion, as such the local empirical loss $f_i(\cdot)$ tends to be similar to each other \cite{hendrikx2020statistically}. 
In this light, we discover that when data held by agents are (close to) homogeneous such that the Hessians are close, i.e., $\grd^2 f_i(\prm) \approx \grd^2 f_j(\prm)$, for any $i,j$, $\prm\in \RR^d$, the transient time for {\algoname} can be significantly shortened. \update{In such setting, the plain {\algoname} algorithm enjoys comparable performance to sophisticated algorithms such as gradient tracking}.  
Concretely, our contributions are summarized as:
\begin{itemize}
\item We present a tight analysis for the expected convergence rate of {\algoname}. Our results focus on revealing the effects of data homogeneity on the convergence rates; see Theorems \ref{thm2-noncvx} and \ref{thm2}. Our analysis relies on a novel use of the high order Taylor approximation of the local gradient maps and exploits the structure of {\algoname} update. The use of this technique can be of independent interest. 
\item We derive improved bounds on transient time for the plain {\algoname} algorithm, i.e., the minimum iteration number required to achieve network independence rate when considering the case of (near) homogeneous data. In particular, we show that the transient time with smooth (possibly non-convex) objective function is $T_{\sf ncvx} = {\cal O}(n^{\frac{2}{3}} /\rho^{\frac{8}{3}})$, and for strongly convex objective function, it is $T_{\sf cvx}={\cal O}(\sqrt{n}/\rho)$. To the best of our knowledge, this is the first improved transient time results for {\algoname} over the existing bounds of $T_{\sf ncvx} = {\cal O}(n^2/\rho^{4})$, $T_{\sf cvx}={\cal O}(n/\rho^2)$ obtained without data homogeneity.
\item We extend the transient time analysis in {\algoname} to study the convergence rate of decentralized {\td} learning with linear function approximation \cite{doan2019finite}. The latter considers the policy evaluation problem for multi-agent Markov decision process, where the data model falls into the homogeneous data setting. Encouragingly, we show that the algorithm in \cite{doan2019finite} enjoys asymptotic network independence and zero transient time, i.e., an asymptotic convergence rate of ${\cal O}(1 / t)$, where $t$ is the iteration number. This improves over the original analysis in \cite{doan2019finite} that shows an \emph{asymptotic} convergence rate of ${\cal O}(\log t /\rho t )$.
\end{itemize}
{Besides, this work presents new proof technique for exploiting a \emph{second order smoothness} property in the analysis of decentralized optimization, which can be of independent interest.}
Finally, we present simulation examples to verify our theoretical guarantees. 
This paper is significantly extended from its preliminary version in \cite{li2022role}. We have included the analysis for strongly convex losses, distributed {\td} learning, and provided an extensive set of numerical experiments.

We remark that the effects of data homogeneity {and the exploitation of high order smoothness property} have been studied for {several sophisticated} algorithms. In detail, \cite{hendrikx2020statistically} derived improved convergence rates for a preconditioned accelerated gradient method with the help of statistical function similarity; \cite{beznosikov2021distributed} studied distributed convex-concave saddle-point problem and proved a lower bound complexity that accounts for data similarity as well as a multi-gossip update algorithm to achieve the lower bound; \cite{sun2022distributed, tian2022acceleration} studied SONATA/ACC-SONATA that utilized partial majorization-minimization and explored data similarity to boost the convergence rate. Compared to the simple {\algoname} algorithm studied in our paper, these works studied algorithms that require either multi-gossip or gradient tracking that adds a considerable complexity to their implementation. Our work is also the first to take high order smoothness as an analysis tool to obtain \emph{tighter bounds} for decentralized optimization.\vspace{.1cm}

\noindent {\bf Paper Organization}:  
\S \ref{sec:problem} introduces the problem structures and plain-{\sf DSGD} algorithm as well as existing transient time bound. \S \ref{sec:main} presents the motivation and our improved theoretical results. In \S \ref{sec:proof_outline}, we shows the proof outline of basic and accelerated rate. In \S \ref{sec:application}, we extend the transient analysis techniques of {\sf DSGD} to decentralized {\td}. Finally, \S \ref{sec:simulation} shows the numerical experiments to validate our analysis. All proofs can be found in the appendix.

\vspace{+.1cm}
\noindent{\bf Notations}: We use $\norm{\cdot}$ to denote the 2-norm of vectors or the matrix spectral norm depending on the argument and $\norm{\cdot}_F$ as the Frobenius norm of matrices. ${\bm 1}$ is the all-one column vector in $\RR^n$. The subscript-less operator $\EE [\cdot]$ denotes the total expectation taken over all randomnesses.

\section{Problem Statement and Assumptions}\label{sec:problem}
 
Our analysis relies on the following assumptions on the loss functions in \eqref{eq:opt}. 

\begin{Assumption}\label{ass:smooth}
For any $i = 1,...,n$, the (local) objective functions satisfy for any $\prm^\prime, \prm \in \RR^d$ that:
\begin{enumerate}[label=(\alph*), nosep]
\item \label{item:lips} there exists $L \geq 0$ such that: 
\beq \label{eq:lips}
\| \grd f_i( \prm^\prime ) - \grd f_i(\prm )\| \leq L \| \prm^\prime  - \prm \|,
\eeq
\item \label{item:hete} there exists $\varsigma \geq 0$ such that:
\begin{align}\label{eq:hete_grad} 
& \| \grd f( \prm ) - \grd f_i( \prm) \| \leq \varsigma,~\forall~\prm \in \RR^d. 
\end{align}
\end{enumerate}
\end{Assumption}
\noindent Notice that \eqref{eq:lips} in A\ref{ass:smooth} is the standard Lipschitz continuous gradient condition that holds for a number of common problems in control and machine learning; while the constant $\varsigma$ in \eqref{eq:hete_grad} bounds the degree of \emph{heterogeneity} between the local objective functions. If $\varsigma = 0$, then the objective functions will be identical to each other (modulo the additive constants).
For the cases of strongly convex objectives, we also consider
\begin{Assumption} \label{ass:str_cvx}
The objective function is $\mu$-strongly convex such that for any $\prm', \prm \in \RR^d$,
\begin{equation}
    f( \prm ) \geq f( \prm' ) + \pscal{ \grd f( \prm' ) }{ \prm - \prm' } + \frac{\mu}{2} \| \prm' - \prm \|^2. 
\end{equation}
\end{Assumption}

The communication graph $G$ is endowed with a weighted adjacency matrix ${\bm W} \in \RR_+^{n \times n}$ satisfying ${\bm W}_{ij} = 0$ iff $(i,j) \notin E$. Moreover, it satisfies 
\begin{Assumption}\label{ass: graph}
The matrix ${\bm W}$ is doubly stochastic, i.e., ${\bm W}\mathbf{1}={\bm W}^\top\mathbf{1}=\mathbf{1}$. There exists a constant ${\rho} \in(0,1]$ and a projection matrix ${\bm U}$ such that $\boldsymbol{I}-\frac{1}{n} \mathbf{1} \mathbf{1}^{\top}={\bm U} {\bm U}^{\top}$ such that $\left\|{\bm U}^{\top} \bm{W} \bm{U} \right\|_{2} \leq 1-{\rho}$.
\end{Assumption}
\noindent The above assumption is standard. It can be satisfied when $G$ is connected and ${\bm W}$ is constructed using the Metropolis-Hasting weights \cite{aldous1995reversible}. Furthermore, the constant $\rho > 0$ is also known as the spectral gap as it measures the connectivity of ${\bm W}$. For ring graph, $\rho={\cal O}(1/n^2)$.

We focus on the classical decentralized stochastic gradient descent (\algoname) algorithm \cite{ram2008distributed, lian2017can}: 
\begin{tcolorbox}[boxsep=2pt,left=4pt,right=4pt,top=3pt,bottom=3pt]
\underline{\textbf{{\algoname} Algorithm:}} Given the initialization $\{ \prm_i^0 \}_{i=1}^n$. For all $t=0,1,2,\ldots$ and $i=1,\ldots,n$,
\beq \label{eq:dsgd} 
\textstyle \prm_i^{t+1} = \sum_{j=1}^{n} {\bm W}_{ij}\prm_j^t - \gamma_{t+1} \grd \ell_i ( \prm_i^t ; Z_i^{t+1} ),
\eeq
where $\gamma_{t}$ is the step size and $\grd \ell_i (\prm_i^t; Z_{i}^{t+1})$ denotes the stochastic gradient taken with respective to $\prm_i^t$ via agent $i$'s current samples $Z_{i}^{t+1}$. 
\end{tcolorbox}
To simplify notations, for all $t \geq 0$, we define
\beq \textstyle 
\Bprm^t := (1/n) \sum_{i=1}^n \prm_i^t,
\eeq 
as the averaged iterate of {\algoname} at iteration $t$ for the rest of this paper.
Furthermore, the stochastic gradients are unbiased and have bounded variance:
\begin{Assumption}\label{ass:SecOrdMom}
For any $i=1,\ldots, n$ and fixed $\prm \in \RR^d$. Let $\prm^\star \in \argmin_{ \prm \in \RR^d } f( \prm )$. It holds $\EE_{Z_{i}\sim {\sf B}_{i}}[\grd \ell_i (\prm; Z_{i}) ] = \grd f_i (\prm)$, there exists $\sigma_0, \sigma_1 \geq 0$ with
\begin{align}
\EE_{Z_i \sim {\sf B}_i} \big[ \| \grd \ell_i(\prm; Z_i) - \grd f_i( \prm ) \|^2 \big] \leq \sigma_0^2 + \sigma_1^2 \norm{ \prm -\prm^{\star} }^2 . 
\notag
\end{align}
\end{Assumption}
\noindent Under A\ref{ass:smooth}, the above condition can be implied by the commonly used bound on the LHS by $\hat{\sigma}_0^2 + \hat{\sigma}_1^2 \norm{ \grd f_i( \prm ) }^2$ \cite{pu2021sharp}. In particular, we note $\norm{ \grd f_i( \prm ) }^2 \leq 2L^2 \| \prm - \prm^\star \|^2 + 2 \| \grd f_i( \prm^\star) \|^2$. 
When $\sigma_1 > 0$, our condition corresponds to a growth condition on the second order moment bound. 

The {\algoname} algorithm \eqref{eq:dsgd} mimics a centralized SGD ({\sf CSGD}) algorithm by performing two operations simultaneously: $\sum_{j=1}^{n} {\bm W}_{ij}\prm_j^t$ performs a consensus update that aggregates decision variables from neighbors, and $\grd \ell_i( \prm_i^t; Z_i^{t+1} )$ is the local stochastic gradient update. 
Our forthcoming analysis will concentrate on the scenarios with  {\sf (a)}~smooth (possibly non-convex) objective function and {\sf (b)}~smooth \& strongly convex objective function. 
We first describe some recent results on the convergence analysis of {\algoname}.

\subsection{Linear Speedup and Transient Time of {\algoname}}
As the {\algoname} algorithm mimics {\sf CSGD} by design, an intriguing question is whether their sampling complexities are comparable. Particularly, as a total of $n$ samples are drawn at each iteration of {\algoname}, one may hope that {\algoname} can achieve the same sampling complexity as {\sf CSGD} \emph{that draws a batch of $n$ samples per iteration}. 

Such phenomena is known as \emph{linear speedup} in the literature, see \cite{lian2017can} for the smooth (possibly non-convex) setting and \cite{pu2021distributed} for the strongly convex setting. 
We observe: 
\begin{Corollary}[Smooth Objective Function]\label{cor:1}
Let $T \geq 1$, set $\gamma_{t+1} = (1 / \sqrt{T}) \sqrt{2 {\sf D} n / (L \sigma^2)}$ and ${\sf T}$ be a random variable (r.v.) chosen uniformly from $\{0,\ldots, T-1\}$. Assume A\ref{ass:smooth}, \ref{ass: graph}, \ref{ass:SecOrdMom}[$\sigma_0 = \sigma, \sigma_1 = 0$]. When $T$ is sufficiently large, we have
    \begin{align}
    \hspace{-.1cm}\EE\left[ \norm{\grd f(\Bprm^{\sf T})}^2 \right] & \leq \sqrt{\frac{32 {\sf D} L \sigma^2 }{n T}} + {\cal O}\left( \frac{n L^2(\varsigma^2+\sigma^2)}{\rho^2 T \sigma^2 } \right). \label{eq:corNor}
    \end{align}
where ${\sf D} := f( \Bprm^0 ) - f^\star$.
\end{Corollary} 
\begin{Corollary}[Strongly Convex Objective Function] \label{cor:2}
    For any $t \geq 1$, set $\gamma_{t}={a_0}/{(a_1+t)}$, $a_0, a_1 > 0$. Assume A\ref{ass:smooth},\ref{ass:str_cvx},\ref{ass: graph},\ref{ass:SecOrdMom}[$\sigma_0 = \sigma_1 = \sigma$], then for any $t \geq 0$
    \beq \label{eq:corStr}
        \EE \left[\normtxt{\Bprm^t -\prm^{\star}}^2\right] \leq 
         {\frac{\sigma^2}{n\mu} \gamma_{t} }
        + {\cal O}\left( \frac{(\sigma^2 + \varsigma^2) L^2 }{\mu^2 \rho^2} \gamma_{t}^2 \right).
    \eeq 
\end{Corollary}
\noindent 
We remark that the use of r.v.~${\sf T}$ in Corollary \ref{cor:1} is a standard setting for analysis of non-convex stochastic optimization, e.g., \cite{ghadimi2013stochastic}. 
For Corollary \ref{cor:2}, we used the step size rule with $\gamma_{t} = {\cal O}(1/t)$ which is also common. See Appendix \ref{app:ncvx1} and \ref{app:cvx1} for the proof of Corollary \ref{cor:1} and \ref{cor:2}, respectively.

Under A\ref{ass:smooth}, \ref{ass:SecOrdMom}, {\sf CSGD} for $\min_{\prm} (1/n) \sum_{i=1}^n f_i(\prm)$ with minibatch size $n$ generates iterates satisfying 
\beq \label{eq:csgdNCvx}
\EE\left[ \norm{\grd f(\prm^{\sf T})}^2 \right] \leq 
\sqrt{\frac{8 {\sf D} L \sigma^2 }{n T}}
=: {\sf UB}_{\sf ncvx}(T,L,{\sf D},\sigma^2)
\eeq 
for smooth (possibly non-convex) optimization \cite{ghadimi2013stochastic} with the step size $\gamma_t = \sqrt{2 n {\sf D} / ( T L \sigma^2 )}$. Moreover,
\beq \label{eq:csgdSCvx}
\EE \left[\normtxt{\prm^t -\prm^{\star}}^2 \right] \leq \frac{\sigma^2}{n \mu} \gamma_{t}=: {\sf UB}_{\sf cvx}(t,\mu,{\sf D},\sigma^2)
\eeq 
\footnote{We have hidden the dependence on $L$ in the constant $\sigma^2$, cf.~A\ref{ass:SecOrdMom}.}for strongly convex optimization \cite{gower2019sgd} with the step size $\gamma_t = a_0 / (a_1+t)$ as in Corollary~\ref{cor:2}.
\update{These bounds match the lower bounds for optimal (centralized)    stochastic gradient algorithms under their respective settings \cite{arjevani2022lower}, \cite{agarwal2012information}.}

Moreover, we observe that ${\sf UB}_{\sf ncvx}, {\sf UB}_{\sf cvx}$ match the dominant terms in \eqref{eq:corNor}, \eqref{eq:corStr} as $t \to \infty$ and {\sf DSGD} \emph{asymptotically} matches the performance of {\sf CSGD} which similarly takes $n$ samples per iteration. Such phenomena is also known as \emph{network independence} since the dominant term is not affected by the network size or connectivity \cite{network_indenp}. 

The central theme of this paper is to study the \emph{transient time} of a decentralized algorithm, defined as the minimum iteration number in which the matching term in {\sf CSGD}, i.e., ${\cal O}( 1/\sqrt{T} )$ in smooth setting, or ${\cal O}( \gamma_t )$ in strongly convex setting, dominates over the remainder terms. Formally, 
\begin{Def}\label{def:ntran}
Under A\ref{ass:smooth}, A\ref{ass: graph}, A\ref{ass:SecOrdMom}[$\sigma_0 = \sigma,  \sigma_1 = 0$], we define the transient time of a decentralized algorithm as
\beq
    T_{\sf ncvx} \!=\! \inf_{T \geq 1} \left\{ T \!:\! \EE[ \normtxt{\grd f(\Bprm^{\sf T})}^2 ] \!\leq\! {\rm c} \, {\sf UB}_{\sf ncvx}(T,L,{\sf D},\sigma^2) \right\} ,
\eeq
where $\Bprm^t := (1/n)\sum_{i=1}^n \prm_i^t$ is the averaged iterate at the $t$th iteration, and ${\sf T}$ is a r.v.~uniformly selected from $\{1,\ldots,T\}$, ${\rm c} \geq 1$ is a constant independent of $T, L , {\sf D}, \sigma$.
\end{Def} 
\begin{Def}\label{def:ctran}
Under A\ref{ass:smooth}, A\ref{ass:str_cvx}, A\ref{ass: graph}, A\ref{ass:SecOrdMom}[$\sigma_0 = \sigma_1 = \sigma$], we define the transient time of a decentralized algorithm as
\beq
T_{\sf cvx} = \inf_{t \geq 1} \left\{ t : \EE [\normtxt{\Bprm^t -\prm^{\star}}^2 ] \leq {\rm c} \, {\sf UB}_{\sf cvx}(t,\mu, {\sf D},\sigma^2)  \right\} ,
\eeq 
where {$\Bprm^t := (1/n)\sum_{i=1}^n \prm_i^t$ is the averaged iterate at the $t$th iteration for the algorithm} and ${\rm c} \geq 1$ is a constant independent of $T, \mu , {\sf D}, \sigma$.
\end{Def}

Corollaries \ref{cor:1} and \ref{cor:2} show that the transient times of {\sf DSGD} via Definitions \ref{def:ntran} and \ref{def:ctran} can be estimated as: 
\beq \label{eq:plain-transient}
\begin{aligned}
 T_{\sf ncvx} & = {\cal O}\left( \frac{n^3 (\sigma^2+\varsigma^2)^2}{\rho^4 \sigma^2 } L^3 \right), \\
 T_{\sf cvx} & = {\cal O}\left(\frac{n (1+\varsigma^2/\sigma^2)}{\rho^2} L^2 \right).
\end{aligned}
\eeq 
In both cases, the transient times grow as ${\cal O}( n^3 / \rho^4 )$ or ${\cal O}( n / \rho^2 )$. When $n$ is large, it may take a long time for {\algoname} to achieve the linear speedup. As discussed in the Introduction, the above observation has motivated prior works \cite{huang2022improving, xin2021improved} to consider sophisticated decentralized algorithm with improved transient time bounds. 

In the sequel, we describe a refined convergence analysis for the {\algoname} algorithm and \emph{tighten} the bounds in Corollaries~\ref{cor:1}, \ref{cor:2}. Our result reveals the role of data homogeneity in which {\algoname} can achieve a fast transient time that is comparable to the state-of-the-art decentralized algorithms.

\begin{remark}
It is worthwhile to mention that our definitions of transient time are based on the network average iterate $\Bprm^t$ similar to \cite{lian2017can}, instead of the local iterate $\prm_i^t$ as in \cite{pu2021sharp}. We remark that such discrepancy can be overcome by running ${\cal O}( \rho^{-1} \log( \epsilon^{-1} ) )$ steps of average consensus (see \cite[Lemma 4]{nedic2009distributed}) to enforce that $\| \prm_i^t - \Bprm^t \|^2 \leq \epsilon$ for any $i \in [n]$, where $\epsilon>0$ is the desired optimality/stationarity gap. 
\end{remark}

\section{Main Results}\label{sec:main}
We first provide an illustrating example that shows {data homogeneity} can strongly influence the transient time of {\algoname}. Consider the following quadratic loss function:
\beq \label{eq:quadratic}
f_i( \prm ) = \EE\big[  (1/2) \prm^\top \widetilde{\bm A}_i \prm + \prm^\top \widetilde{\bm b}_{i} \big],
\eeq 
where the expectation is taken w.r.t.~$\widetilde{\bm A}_i, \widetilde{\bm b}_i$ such that $\EE[ \widetilde{\bm b}_{i} ] = {\bm b}_i$ and $\EE[ \widetilde{\bm A}_{i} ] = {\bm A}_i$ is a symmetric positive definite matrix. 

Assume homogeneous data such that there exists a common positive definite matrix ${\bm A}$ shared among the agents with ${\bm A}_i = {\bm A}$. Consider the following stochastic gradient map 
\beq\label{eq:stoc_grad_quad}
    \grd \ell(\prm; Z_{i}) = \widetilde{\bm A}_i \prm + \widetilde{\bm b}_i,
\eeq
where $Z_{i} = (\widetilde{\bm A}_i, \widetilde{\bm b}_i)$ are independent random variable. The variances are assumed as bounded with $\EE[\normtxt{\widetilde{\bm A}_i - {\bm A}}^2] \leq \sigma^2$, $\EE[ \| \widetilde{\bm b}_i - {\bm b}_{i} \|^2 ] \leq \sigma^2$.
This implies 
\begin{align*}
    \EE[ \| \grd \ell ( \prm_i^t; Z_{i} ) - \grd f_i (\prm_i^t) \|^2 ] \lesssim \sigma^2(1+\norm{\prm_i^t - \thstr}^2),
\end{align*} 
for any $\prm_i^t \in \RR^d$ and $\lesssim$ hides the numerical constants that are independent of $\sigma^2$.

With \eqref{eq:stoc_grad_quad}, the {\sf DSGD} algorithm reads:
\begin{align} 
& \textstyle \prm_i^{t+1} = \sum_{j=1}^n {\bm W}_{ij} \prm_i^t - \gamma_{t+1} \big( \widetilde{\bm A}_i \prm_i^t  + \widetilde{\bm b}_i \big). \label{eq:dsgd_quad} 
\end{align}
Using the fact that $\sum_{i=1}^n {\bm W}_{ij} = 1$, the averaged iterates are updated as
\beq \label{eq:linearkey}
\begin{aligned} 
\Bprm^{t+1} & \textstyle  = \Bprm^t - {\gamma_{t+1}} \big(
\sum_{i=1}^{n} \widetilde{\bm A}_i \prm_i^t/n  + \sum_{i=1}^n \widetilde{\bm b}_i  /n \big). 
\end{aligned}
\eeq 
Importantly, due to $\EE[ \widetilde{\bm A}_i ] = {\bm A}$, the last term is an \emph{unbiased estimate} of the gradient of $f(\prm) = \frac{1}{n}\sum_{i=1}^n f_i(\prm)$ with 
\[ 
\textstyle \EE[ {n}^{-1} \sum_{i=1}^n \big( \widetilde{\bm A}_i \prm_i^t  + \widetilde{\bm b}_i \big) ] = \grd f( \Bprm^t ).
\]
The variance of the gradient estimator is bounded by
{\small 
\beqq 
    \EE\Bigg[ \Big\| \frac{1}{n} \sum_{i=1}^n (\widetilde{\bm A}_i\prm_i^t + \widetilde{\bm b}_i )-  \grd f( \Bprm^t ) \Big\|^2 \Bigg]
    \!\lesssim\! \frac{\sigma^2}{n}\left[1+\sum_{i=1}^{n}\frac{ \normtxt{\prm_i^t-\thstr}^2 }{n} \right],
\eeqq}where $\lesssim$ omits numerical constant for the upper bound. In comparison, the {\sf CSGD} algorithm applied to $f(\prm)$ with a batch size $n$ admits a variance of ${\cal O}( \frac{\sigma^2}{n} (1+\| \Bprm^t - \thstr\|^2) )$. We observe that the only difference between \eqref{eq:linearkey} and the {\sf CSGD} algorithm lies in the extra error term in the \emph{variance} bound due to the consensus error $\sum_{i=1}^n \| \prm_i^t - \Bprm^t \|^2$. We anticipate the transient time for {\algoname} to be much less than \eqref{eq:plain-transient}.

We generalize the above example to non-quadratic objective functions by considering the following set of additional conditions.
First, notice that one of the keys to accelerating the transient time lies in the similarity between the \emph{Hessians} of the objective functions. To this end, we impose: 
\begin{Assumption}\label{ass:hete2}
There exists $\varsigma_H \geq 0$ such that for any $i=1,\ldots, n$,
\begin{align}
& \| \grd^2 f( \prm ) - \grd^2 f_i( \prm) \| \leq \varsigma_H,~\forall~\prm \in \RR^d. \label{eq:hete_hess}
\end{align}
\end{Assumption}
\noindent The constant $\varsigma_H$ quantifies the similarity between the Hessians of the component function $f_i(\prm)$. \update{While both A\ref{ass:smooth}-\ref{item:hete} and A\ref{ass:hete2} hold under the setting of homogeneous data, we note that $\varsigma = 0$ in A\ref{ass:smooth}-\ref{item:hete} implies $\varsigma_H = 0$ in A\ref{ass:hete2}, but not vice versa.}
{Furthermore, as shown in \cite{hendrikx2020statistically}, for empirical risk minimization (ERM) problems with $m$ i.i.d.~data samples split across the agents, one has $\varsigma_H = {\cal O}( 1 / \sqrt{m} )$.}

We also require the following technical assumptions:
\begin{Assumption} \label{ass:Hsmooth}
There exists $L_H \geq 0$ such that for any $i=1,\ldots, n$,
\beq
\| \grd^2 f_i( \prm' ) - \grd^2 f_i( \prm ) \| \leq L_{H} \| \prm' - \prm \|,~\forall~\prm, \prm' \in \RR^d.
\eeq
\end{Assumption} 
\begin{Assumption}\label{ass:SecOrdMom2}
For any $i=1,\ldots, n$ and fixed $\prm \in \RR^d$. Let $\prm^\star \in \argmin_{ \prm \in \RR^d } f( \prm )$. There exists $\bar\sigma_0, \bar\sigma_1 \geq 0$ with
\begin{align}
\EE_{Z_i\sim {\sf B}_{i}} [ \| \grd \ell( \prm; Z_{i}) \!-\! \grd f_i( \prm ) \|^4 ]  \leq \bar\sigma_0^4 + \bar\sigma_1^4 \norm{ \prm - \prm^\star }^4,\notag
\end{align}
and it holds that $\EE_{Z_i \sim {\sf B}_i} [ \grd \ell_i( \prm; Z_i ) ] = \grd f_i( \prm )$.
\end{Assumption}
\noindent For quadratic functions, A\ref{ass:Hsmooth} is satisfied with $L_H = 0$; {the assumption is further satisfied with common loss functions such as the logistics loss, see \cite{wai2020accelerating}}. Meanwhile, A\ref{ass:SecOrdMom2} controls the $4$th order moment bound on the variance akin to A\ref{ass:SecOrdMom}. In fact, both A\ref{ass:SecOrdMom}, A\ref{ass:SecOrdMom2} are consequences of an almost sure bound on the gradient noise with the growth condition \cite{bottou2018optimization} $\sup_{z \in {\rm supp}( {\sf B}_i )} \| \grd \ell( \prm; z) \!-\! \grd f_i( \prm ) \| \leq \tilde{\sigma}_0 + \tilde{\sigma}_1 \| \grd f_i (\prm) \|$. The latter holds in the finite-sum optimization setting or in learning problems with bounded data.
\subsection{Smooth (possibly non-convex) case}
We will first present the general convergence rate analyzed under the additional conditions (A\ref{ass:hete2}--A\ref{ass:SecOrdMom2}) \update{emphasizing on the role of data homogeneity}, then we will discuss its implication on the transient time analysis. 
\begin{theorem}\label{thm2-noncvx}
Under A\ref{ass:smooth},\ref{ass: graph},\ref{ass:SecOrdMom}[$\sigma_0 = \sigma, \sigma_1 = 0$],\ref{ass:hete2},\ref{ass:Hsmooth},\ref{ass:SecOrdMom2}[$\bar\sigma_0 = \bar\sigma, \bar\sigma_1 = 0$], suppose that step size satisfies $\sup_{t\geq 1}\gamma_{t} \!\leq\!  \frac{\rho}{10L\sqrt[4]{n}}$. 
Let {${\sf D} := f( \Bprm^0 ) - f^\star$}. For any $T \geq 1$, it holds 
\begin{align*}
    & \EE \left[ \sum_{t=0}^{T-1}\gamma_{t+1}\norm{\grd f(\Bprm^t)}^2 \right] \leq 4 {\sf D} + \frac{2L\sigma^2}{n} \sum_{t=0}^{T-1} \gamma_{t+1}^2 \\
    & + \frac{ {432} L_H^2}{\rho^4}(\bar\sigma^4+4\varsigma^4) \sum_{t=0}^{T-1}\gamma_{t+1}^5 + \frac{ {64} \varsigma_H^2 (\sigma^2 + \varsigma^2) }{\rho^2} \sum_{t=0}^{T-1} \gamma_{t+1}^3 \\
    & + \frac{ 4 \gamma_{1} }{ \rho n^2 } \left( 4 n\varsigma_H^2 \| \CSE{0} \|_F^2 + L_H^2 \| \CSE{0} \|_F^4 \right) .
 \end{align*}
\end{theorem}
\noindent The above theorem provides a tighter characterization for the convergence of {\algoname} than Corollary~\ref{cor:1}. Notably, under a similar step size condition, the dominant term remains comparable to ${\sf UB}_{\sf ncvx}(T,L,{\sf D},\sigma^2)$, while the transient term is decomposed into a slower one that depends on $\varsigma_H^2$, and a faster one that depends on $L_H^2, \sigma^2, \varsigma^2$. Note that this explicitly accounts for the effects of data homogeneity via $\varsigma_H^2$. 

\begin{tcolorbox}[boxsep=2pt,left=4pt,right=4pt,top=3pt,bottom=3pt]
\begin{Corollary}\label{cor:3}
Set $\gamma_{t+1} = (1 / \sqrt{T}) \sqrt{2 {\sf D} n / (L \sigma^2)}$ and let ${\sf T}$ be chosen uniformly at random from $\{0,\ldots, T-1\}$. 
Consider the same set of assumptions as in Theorem~\ref{thm2-noncvx}. The following holds for sufficiently large $T \geq 1$,
\begin{align}
& \EE \left[ \big \| \grd f(\Bprm^{\sf T}) \big\|^2 \right] \leq  
\sqrt{\frac{32 {\sf D} L \sigma^2 }{n T}}
\label{eq:corHOS} \\
& + \frac{ {432} L_H^2(\bar\sigma^4 + 4\varsigma^4)}{\rho^4 T^{2} } \frac{ (2 {\sf D} n)^2 }{ (L \sigma^2)^2 } + \frac{ {64} \varsigma_H^2 (\sigma^2 + \varsigma^2) }{\rho^2 T} \frac{ 2 {\sf D} n }{ L \sigma^2 }. \notag 
\end{align}
\end{Corollary}
\end{tcolorbox}
\noindent 
Corollary~\ref{cor:3} shows that the transient time of {\algoname} is: 
\beq \label{eq:b2}
    T_{\sf ncvx} = 
    {\cal O} \left( \frac{ n^{ \frac{5}{3} } }{ \rho^{ \frac{8}{3} } } \frac{ L_{H}^{\frac{4}{3}} (\bar\sigma^{\frac{8}{3}} + \varsigma^\frac{8}{3}) }{ L^{ \frac{5}{3} } \sigma^{ \frac{10}{3} } } + \frac{n^3}{\rho^4} \frac{\varsigma_H^4 (\sigma^4 + \varsigma^4) }{ L^3 \sigma^4 }  \right) .
\eeq
We now concentrate on the case of \emph{homogeneous data} where $\varsigma_H \approx 0$ such that the second term in \eqref{eq:b2} becomes negligible.
Under this approximation, \eqref{eq:b2} gives a transient time of ${\cal O}( L_H^{\frac{4}{3}} n^{\frac{5}{3}} / \rho^{ \frac{8}{3} } )$ which improves over \eqref{eq:plain-transient} with ${\cal O}( L^3 n^3/\rho^4)$ in terms of the dependence on $n, \rho$.

\subsection{Strongly convex case}
Similar to the previous subsection, we will present the general convergence rate under A\ref{ass:hete2}--A\ref{ass:SecOrdMom2}, then discuss about the potential acceleration to transient time \update{due to data homogeneity}. 
Assume that $f(\prm)$ is a strongly convex function on $\RR^d$. 
Denote the optimal solution of \eqref{eq:opt} as $\prm^\star \eqdef \argmin_{\prm\in \RR^d} f(\prm).$ 

Under this setting, we observe:
\begin{theorem}\label{thm2}
Under A\ref{ass:smooth}, \ref{ass:str_cvx}, \ref{ass: graph}, \ref{ass:SecOrdMom}[$\sigma_0 = \sigma_1 = \sigma$], \ref{ass:hete2}, \ref{ass:Hsmooth},\ref{ass:SecOrdMom2}[$\bar{\sigma}_{0}=\bar{\sigma}_{1}=\bar{\sigma}$], for fixed parameter $\tdelta>0$, and step size $\gamma_t = {a_0}/({a_1+t})$, where $a_0, a_1\in\RR^+$. Suppose the step size $\{\gamma_{t}\}_{t\geq 1}$ satisfies
\begin{align*}
    & \sup_{t\geq 1}\gamma_{t} \leq \min\left\{ \sqrt[3]{\frac{\mu}{8 c_2}} ,\frac{\rho}{\mu},
    \frac{\rho}{2\sqrt{\sigma^2 + 2L^2}}, \frac{\mu}{8(\sigma^2 +L^2)}, \frac{\rho}{\sqrt[4]{c_3}}
    \right\}
    \\
    &{\gamma_{t-1}}/\gamma_{t} \leq \sqrt{1+ (\mu/4) \gamma_{t}^r}, ~~ \forall t\geq 1, \forall r\in\{2,3,4,5\},
\end{align*}
Then, the following bound holds with probability at least $1-\frac{a_0^2}{a_1}\tdelta$: for any $t\geq 0$,
\begin{align}\label{eq_thm2_mse} 
    &\EE\left[\big\|\Bprm^{t} - \prm^\star \big\|^2\right] \leq 
    \prod_{i=1}^{t} \left(1- \frac{\mu}{4} \gamma_{i} \right){\sf D}' +
    \frac{16\sigma^2}{n \mu} \gamma_{t} \notag
    \\
    & + \frac{128 (\sigma^2 +\varsigma^2) \varsigma_H^2}{\mu^2 \rho^2}\gamma_{t}^2 + \frac{768 }{\mu\rho^2}(\sigma^2+L^2)(\sigma^2 +\varsigma^2) \gamma_{t}^3 \notag
     \\
     &  + \frac{864 }{\mu^2\rho^4}L_{H}^2(\bar{\sigma}^4+4\varsigma^4)\gamma_{t}^4 ,
\end{align}
where {$c_1 \eqdef 4(\sigma^2+L^2)$}, $c_2$, $c_3$, ${\sf D}'$ are constants defined as
\begin{align*}
    c_2 &\eqdef \textstyle \frac{192 \sigma^2}{\rho^2} \left( \frac{\bar{\sigma}^4 L_H^2}{\tdelta n \mu^2 \rho^2} + \sigma^2 + L^2 \right), 
    c_3 \eqdef 864 n (\bar{\sigma}^4 + 8 L^4) 
    \\
    {\sf D}' &\eqdef \textstyle \EE \left[ \normtxt{\Bprm^{0} - \prm^\star }^2 
    + \frac{4 c_{1}}{n \rho}\gamma_{1}^2 \normtxt{\CSE{0}}_F^2  + \frac{2L_H^2}{\mu\rho n^2}\gamma_{1}  \norm{\CSE{0}}_F^4\right]
\end{align*}
\end{theorem}
\noindent Similar to Theorem~\ref{thm2-noncvx}, the above theorem offers a tightened bound for the convergence of {\algoname} than Corollary~\ref{cor:2}. The dominant term ${\cal O}(\gamma_t)$ remains comparable to ${\sf UB}_{\sf cvx}(t,\mu, {\sf D},\sigma^2)$. Meanwhile, the dominating transient term of order ${\cal O}(\gamma_t^2)$ vanishes as $\varsigma_H^2 \to 0$ and the remaining terms that reflect the network topology $\rho$ and data homogeneity $\varsigma$ is now of the order ${\cal O}(\gamma_t^3)$. 

We remark that Theorem \ref{thm2} has been presented under a diminishing step size rule $\gamma_t = \frac{ a_0 }{ a_1 + t }$. This choice is made for the sake of simplicity only.
The expected convergence bound holds under the condition of {an event} occurring with probability at least $1-\tilde{\delta}$. This is an artifact due to the need to analyze high order moments of the optimality gap $\EE[ \| \Bprm^t - \prm^\star \|^4 ]$; see Sec.~\ref{sec:pf_thm_scvx}. Simplifying \eqref{eq_thm2_mse} gives

\begin{tcolorbox}[boxsep=2pt,left=4pt,right=4pt,top=3pt,bottom=3pt]
\begin{Corollary}\label{cor:cvx_tran}
Consider the same set of assumptions as in Theorem~\ref{thm2}. The following holds with high probability,
    \begin{align*} 
    & \EE[ \| \Bprm^t - \prm^\star \|^2 ] \\
    & = {\cal O} \left(\frac{\sigma^2}{n \mu} \frac{1}{t}  + \frac{(\sigma^2+L^2)(\sigma^2 +\varsigma^2)}{\mu\rho^2} \frac{1}{t^3}  + L_{H}^2 \frac{\bar{\sigma}^4+\varsigma^4}{\mu^2\rho^4} \frac{1}{t^4}  \right) \\
    & \quad + {\cal O} \left( \varsigma_H^2 \frac{ \sigma^2 +\varsigma^2 }{ \mu^2 \rho^2 } \frac{1}{t^2} \right) .
    \end{align*}
\end{Corollary}
\end{tcolorbox}
\noindent From Corollary \ref{cor:cvx_tran}, we deduce that the transient time [cf.~Definition \ref{def:ctran}] for {\algoname} under the above premises is given by
\begin{align} \label{eq:transient-scvx-improved}
T_{\sf cvx} = {\cal O}\left( \frac{\sqrt{n}(1+L/\sigma)(\sigma+\varsigma)}{\rho} + \frac{\varsigma_H^2 n}{\rho^2} \cdot \frac{\sigma^2 +\varsigma^2}{\mu \sigma^2} \right).
\end{align}
{Again, to study \eqref{eq:transient-scvx-improved}, we concentrate on the \emph{homogeneous data setting} where $\varsigma_H \approx 0$ and the second term can be ignored.}
We recall that the previously known bound for transient time of {\algoname} is ${\cal O}(n / \rho^2)$ as shown in \eqref{eq:plain-transient}; also see \cite{pu2021distributed}. Our bound improves it to ${\cal O}( \sqrt{n} / \rho )$ under the additional A\ref{ass:hete2}--A\ref{ass:SecOrdMom2}. 
We show that the {\algoname} algorithm takes advantage of homogeneous data to yield an accelerated transient time.

\section{Proof of Main Results}\label{sec:proof_outline}
To setup the analysis, we observe that the average iterate, $\Bprm^t = (1/n) \sum_{i=1}^n \prm_i^t$,  satisfies the recursion for any $t \geq 0$:
\begin{align}\label{eq:avg_it} \textstyle 
    \Bprm^{t+1} = \Bprm^t - (\gamma_{t+1} /n) \sum_{i=1}^{n} \grd \ell_i ( \prm_i^t ; Z_i^{t+1} ) .
\end{align}
We shall use $\EE_t [\cdot]$ to denote the expectation conditional on the filtration given by the sigma-field $\sigma( \prm_i^0, Z_i^s, s \leq t, \forall i)$.
To facilitate our discussions, we define the quantities
\begin{align}
& \widetilde{\boldsymbol{\theta}}^t:=\overline{\boldsymbol{\theta}}^t-\boldsymbol{\theta}^\star,~~\boldsymbol{\Theta}_o^t:=\left(\boldsymbol{\theta}_1^t \cdots \boldsymbol{\theta}_n^t\right)-\overline{\boldsymbol{\theta}}^t \mathbf{1}^{\top},\label{eq:constant}
\end{align}
From \eqref{eq:avg_it}, it is clear that the {\algoname} recursion differs from that of a {\sf CSGD} algorithm only by the iterate that the stochastic gradient map evaluates on. The latter would update $\Bprm^{t}$ through the direction $(1/n) \sum_{i=1}^{n} \grd \ell_i ( \Bprm^t ; Z_i^{t+1} )$. 

The above observation suggests that the analysis of {\algoname} hinges on how to control the difference $\sum_{i=1}^{n} ( \grd \ell_i ( \prm_i^t ; Z_i^{t+1} ) - \grd \ell_i ( \Bprm^t ; Z_i^{t+1} ) )$, which has the expected value $\sum_{i=1}^{n} ( \grd f_i ( \prm_i^t ) - \grd f_i ( \Bprm^t ) )$. Under A\ref{ass:smooth}, the latter may be bounded as $L \sum_{i=1}^n \| \prm_i^t - \Bprm^t \|$. 

We depart from using such a crude bound obtained by A\ref{ass:smooth}. The key technique in our refined analysis is to study the expected difference vector via Taylor expansion and the second order smoothness property A\ref{ass:Hsmooth}. Consider the following approximation error for the gradient map $\grd f_i(\prm)$:
\begin{align}
    {\cal M}_i (\prm^\prime; \prm) \eqdef \grd f_i(\prm^{\prime}) - \update{\grd} f_i(\prm) - \grd^2 f_i(\prm)(\prm^\prime-\prm)
\end{align}
By A\ref{ass:Hsmooth}, it holds that \cite[Lemma 1.2.5]{nesterov2003introductory}
\begin{align}\label{eq:nesterov}
    \normtxt{{\cal M}_i (\prm^\prime; \prm)} \leq \frac{L_{H}}{2} \norm{\prm^\prime - \prm}^2, \forall \prm, \prm^\prime \in \RR^d.
\end{align}
Importantly, it can be derived that 
\begin{align}\label{eq-is}
& \textstyle \sum_{i=1}^{n} [\grd f_i (\prm_i^t) - \grd f_i(\Bprm^t)]
\\ 
& \textstyle = \sum_{i=1}^{n} \left( {\cal M}_i(\prm_i^t; \Bprm^t) + [\grd^2 f_i(\Bprm^t)-\grd^2 f(\Bprm^t)](\prm_i^t - \Bprm^t) \right)\nonumber,
\end{align}
where we have used the linearity property $\frac{1}{n}\sum_{i=1}^{n} \grd^2 f(\Bprm^t)\prm_i^t =  \grd^2 f(\Bprm^t)\Bprm^t$. When $\varsigma_H = 0$, the last term on the RHS of \eqref{eq-is} vanishes and \begin{equation} 
\textstyle \left\| \sum_{i=1}^{n} [\grd f_i (\prm_i^t) - \grd f_i(\Bprm^t)] \right\| \leq \frac{L_H}{2} \| \Prm_o^t \|_F^2.
\end{equation} 
This yields a \emph{quadratic upper bound} that decays faster than the crude bound from A\ref{ass:smooth}. In the remainder of this section, we develop tighter bounds for the convergence of {\algoname} utilizing the above observation.

\subsection{Proof of Theorem~\ref{thm2-noncvx}}
Taking the insights from \eqref{eq-is}, we observe the improved descent lemma:
\begin{lemma}\label{lem:des2-ncvx}
Under A\ref{ass:smooth},\ref{ass: graph},\ref{ass:SecOrdMom}[$\sigma_{0}=\sigma, \sigma_{1}=0$],\ref{ass:hete2},\ref{ass:Hsmooth}, if $\sup_{t \geq 1} \gamma_{t} \leq 1/(4L)$, 
then for any $t \geq 0$, it holds
    \beq \label{eq:des}
    \begin{aligned}
        \EE_t[ f( \Bprm^{t+1} ) ] & \leq f( \Bprm^t ) - \frac{\gamma_{t+1}}{4} \| \grd f( \Bprm^t ) \|^2 + \frac{ \gamma_{t+1}^2 L \sigma^2}{2n} \\
        & \quad + {\frac{\gamma_{t+1}}{n} }\Big(  \frac{L_{H}^2}{n} \| \CSE{t} \|_F^4 + 4 \varsigma_H^2 \| \CSE{t} \|_F^{2} \Big) .
    \end{aligned}
    \eeq
\end{lemma}
\noindent The proof is in Appendix \ref{app:ncvx2}. The last term of \eqref{eq:des} manifests the effects of second order smoothness deduced in \eqref{eq-is}. In particular, the \emph{perturbation} term in the above descent lemma is proportional to $\varsigma_H^2 \| \CSE{t} \|_F^{2} + \frac{L_H^2}{n} \| \CSE{t} \|_F^4$, whose first term vanishes as $\varsigma_H \to 0$ and the second term is anticipated to decay at a fast rate. 

To this end, we observe the following bounds on $\| \CSE{t} \|_F^2, \| \CSE{t} \|_F^4$. 
\begin{lemma}\label{lem:ceb-ncvx} Under A\ref{ass:smooth},\ref{ass: graph},\ref{ass:SecOrdMom}[$\sigma_0 = \sigma,  \sigma_1 = 0$]. If $\sup_{t \geq 1}\gamma_{t} \leq \frac{\rho}{4L}$, then for any $t \geq 0$, it holds
\begin{align}
& \EE_t [\norm{\CSE{t+1}}_F^2] \leq \left(1- \frac{\rho}{2} \right)  \norm{\CSE{t}}_F^2 + 2n (\varsigma^2 + \sigma^2 ) \frac{\gamma_{t+1}^2}{\rho}
\end{align}
If in addition A\ref{ass:SecOrdMom2}[$\bar{\sigma}_{0}=\bar{\sigma}, \bar{\sigma}_{1}=0$] holds, $\sup_{t \geq 1}\gamma_{t}\leq \frac{\rho}{10 L \sqrt[4]{n}} $, then for any $t \geq 0$, it holds
\begin{align}
\EE_t \! \norm{\CSE{t+1}}^4_F \leq \left( 1 - \frac{\rho}{2} \right) \norm{\CSE{t}}_F^4  \!+\! 54n^2 (\bar{\sigma}^4 \!+\! 4 \varsigma^4)\frac{\gamma_{t+1}^4}{\rho^3}
\end{align}
\end{lemma}
\noindent The proof is in Appendix \ref{app:unif-cons-bd}. Subsequently, we construct the Lyapunov function:
\beq 
{\tt V}^t := \EE \Big[ f( \Bprm^t ) - f^\star + \frac{ {2} \gamma_{t} }{ \rho n } \big( \textstyle 4 \varsigma_H^2 \| \CSE{t} \|_F^2 + \frac{L_H^2}{n} \| \CSE{t} \|_F^4 \big) \Big] .
\eeq 
Note that ${\tt V}^t \geq 0$ for $t \geq 0$. Combining Lemma~\ref{lem:des2-ncvx}, \ref{lem:ceb-ncvx} shows
\beq 
\begin{aligned}
 {\tt V}^{t+1}  & \leq {\tt V}^t - \frac{ \gamma_{t+1} }{4} \| \grd f( \Bprm^t ) \|^2 + \frac{ \gamma_{t+1}^2 L \sigma^2}{2n} \\
& + {16} \varsigma_H^2 (\varsigma^2 + \sigma^2 ) \frac{\gamma_{t+1}^3}{\rho^2} + {108} L_H^2  (\bar{\sigma}^4 \!+\! 4 \varsigma^4)\frac{\gamma_{t+1}^5}{\rho^4}  
\end{aligned}
\eeq 
Summing up the above inequality from $t=0$ to $t=T-1$ and rearranging terms lead to Theorem~\ref{thm2-noncvx}.  

\subsection{Proof of Theorem~\ref{thm2}} \label{sec:pf_thm_scvx}
Our analysis relies on the following refined descent lemma under strongly convex objective function. 
\begin{lemma}\label{lem:des2-scvx}
Under A\ref{ass:smooth},\ref{ass:str_cvx},\ref{ass:SecOrdMom}[$\sigma_{0}=\sigma_{1}=\sigma$],\ref{ass:hete2},\ref{ass:Hsmooth}. Assume that the step size satisfies $\sup_{t \geq 1} \gamma_{t}\leq {\mu}/(8(\sigma^2 + L^2))$. Then, it holds for any $t \geq 0$ that
\begin{align}\label{eq:ba}
    &\EE_t\norm{\Tprm^{t+1}}^2 
    \leq  (1-\frac{\mu}{2} \gamma_{t+1}) \normtxt{\Tprm^{t}}^2 + {\frac{2\sigma^2}{n}}\gamma_{t+1}^2 
    \\
    &\! + \! \frac{ \gamma_{t+1}^2 }{ n } 4(\sigma^2 \!+\! L^2)  \norm{\CSE{t}}_F^2 + \frac{ \gamma_{t+1} }{2 n \mu} \Big(  \frac{L_{H}^2}{n} \| \CSE{t} \|_F^4 + 4 \varsigma_H^2 \| \CSE{t} \|_F^{2} \Big) .  \nonumber
\end{align}
\end{lemma}
\noindent See the proof in Appendix~\ref{app:cvx2}.

We note the last row gathers the \emph{perturbation} terms due to the consensus error. The first term depends on $\gamma_{t+1}^2$ such that it can be controlled by making the step size small; the second term is proportional to $\varsigma_H^2 \norm{\CSE{t}}_F^2 + \frac{L_H^2}{n} \norm{\CSE{t}}_F^4$, where the former term vanishes as $\varsigma_H \to 0$ and the latter term is anticipated to decay quickly w.r.t.~consensus error. 

Our next step is to observe the following bounds on the consensus error:
\begin{lemma}\label{lem:ceb-scvx} Under A\ref{ass:smooth},\ref{ass: graph},\ref{ass:SecOrdMom}[$\sigma_0 = \sigma_1 = \sigma$]. If $\sup_{t\geq 0} \gamma_{t}\leq {\rho}/{\sqrt{8(\sigma^2 + 2L^2)}},$ then for any $t\geq 0$, it holds that 
    \begin{eqnarray}
        \begin{aligned}
            &\EE_t \norm{\CSE{t+1}}_F^2 
            \leq (1-\frac{\rho}{2})\norm{\CSE{t}}_F^2 
            \\
            & + \frac{ 2n\gamma_{t+1}^2 }{\rho} \Big[(\sigma^2 + \varsigma^2) +2\sigma^2 \normtxt{\Tprm^t}^2\Big].
        \end{aligned}
    \end{eqnarray}
    If in addition A\ref{ass:SecOrdMom2}[$\bar{\sigma}_{0}= \bar{\sigma}_1 = \bar{\sigma}$] holds and $\sup_{k\geq 1}\gamma_{k} \leq {\rho}/{\sqrt[4]{ c_3}}$, where $c_3\eqdef 864 n (\bar{\sigma}^4 + 8 L^4)$.
    Then, it holds 
    \begin{eqnarray}
       \begin{aligned}
            &\EE_t \norm{\CSE{t+1}}^4_F 
            \leq (1-\frac{\rho}{2})\norm{\CSE{t}}_F^4 
            \\
            & + \frac{ 54 n^2 \gamma_{t+1}^4}{ \rho^3 }  \left[ \bar{\sigma}^4 + 4\varsigma^4 +8\bar{\sigma}^4 {\normtxt{\Tprm^t}^4} \right] 
        \end{aligned}
    \end{eqnarray}
\end{lemma}
\noindent The proof is in Appendix \ref{app:unif-cons-bd}. Note that the above differs from Lemma~\ref{lem:ceb-ncvx} only through the growth condition with non-zero $\sigma_1, \bar{\sigma}_1$ on the second/fourth order moment of the stochastic oracle [cf.~A\ref{ass:SecOrdMom} \& \ref{ass:SecOrdMom2}].

At this point, we may proceed by analyzing the following Lyapunov function
\beq \label{eq:lyap-fct-main}
\begin{aligned}
    {\tt L}_{t} & \eqdef \EE\Big[ \normtxt{\Tprm^{t}}^2 
        + \frac{4 c_{1}}{n \rho}\gamma_{t}^2 \normtxt{\CSE{t}}_F^2 \\
        & \qquad + \frac{8\varsigma_{H}^2}{n \mu \rho} \gamma_{t}\norm{\CSE{t}}_F^2 + \frac{2L_H^2}{\mu\rho n^2}\gamma_{t}  \norm{\CSE{t}}_F^4 \Big].
\end{aligned} 
\eeq 
Unlike the proof of Theorem~\ref{thm2-noncvx}, we observe that directly combining Lemma~\ref{lem:des2-scvx}, \ref{lem:ceb-scvx} does not lead to the desired bound. In particular, for any $t \geq 0$ {and sufficient small $\gamma_{t}$}, it holds that
\beq \label{eq:lt}
\begin{aligned}
\EE_t [{\tt L}_{t+1}] & \leq \left(1 - \gamma_{t+1} \frac{\mu}{4} + \gamma_{t+1}^5 {\frac{ 864 L_H^2 \bar{\sigma}^4 }{ \mu \rho^4} } \| \Tprm^t \|^2 \right) {\tt L}_t \\
& + \frac{2\sigma^2}{n}\gamma_{t+1}^2 + {\rm D}\varsigma_H^2 \gamma_{t+1}^3 + {\rm E}\gamma_{t+1}^4 + {\rm F}\gamma_{t+1}^5 \eqsp,
\end{aligned} 
\eeq   
where ${\rm D}, {\rm E}, {\rm F}$ are constants defined as
\beq 
\label{eq:constant_scvx_def}
\begin{aligned}
    {\rm D} &\eqdef \frac{ 16(\sigma^2 + \varsigma^2)}{\mu \rho^2}, \quad {\rm E} \eqdef \frac{96(\sigma^2+\varsigma^2)(\sigma^2+L^2)}{\rho^2}, \\
    {\rm F} &\eqdef \frac{108(\bar{\sigma}^4 +4\varsigma^4)L_{H}^2}{\mu\rho^4}.
\end{aligned}
\eeq 

The main challenge for analyzing \eqref{eq:lt} is that the bound on the RHS involves a fourth order term, $\| \Tprm^t \|^4$, as we recall that ${\tt L}_t$ contains $\| \Tprm^t \|^2$. Our remedy is to obtain a high probability bound for the random variable $\| \Tprm^t \|^2$ and use it to control the contraction in the first term of \eqref{eq:lt}. Note that such bound does not need to be tight as $\| \Tprm^t \|^2$ since it will be controlled by the 5-th order step size term $\gamma_{t+1}^5$. To this end, we can easily derive such bound using Corollary~\ref{cor:1} and Markov inequality:

\begin{lemma}\label{lem:hp-bnd}
Assume A\ref{ass:smooth}-\ref{ass:SecOrdMom}[$\sigma_0=\sigma_1=\sigma$] and the step size $\gamma_{t}=a_{0}/(a_1 + t)$. Then, for any $\tdelta>0,$ it holds
\beq\label{eq:hp-bnd}
\normtxt{\Tprm^{t}}^2 \leq (\tdelta\gamma_{t})^{-1} {{\rm C}},\quad \forall t \geq \max\{t_0, t_1\},
\eeq
    with probability at least $1-\tdelta a_0^2/a_1$, where $a_0, a_1\in \RR_{+}$, ${\rm C} \eqdef {32\sigma^2}/{(n\mu)}$ and $t_0, t_1$ are defined in \eqref{eq:t0} and \eqref{eq:t1}.
\end{lemma}
\noindent See the proof in Appendix~\ref{app:hp-bnd}.

Using (\ref{eq:hp-bnd}), we observe that $\gamma_{t+1}^5 \| \Tprm^t \|^2 = {\cal O}(\gamma_{t}^4)$ for any finite $t$ w.h.p.. Choosing a proper step size immediately lead the bracket term in \eqref{eq:lt} to be strictly smaller than $1$, i.e., it contracts. Applying Lemma \ref{lem:hp-bnd} with the step size condition 
\[\textstyle
    {\gamma_{t+1}}/{\gamma_{t}} \leq \min\bigg\{ \sqrt{1+(\mu/4)\gamma_{t}^{p}}, \forall p \in \{2,3,4,5\}\bigg\},~ t\geq 1
\]
leads to Theorem \ref{thm2}; see details in Appendix \ref{app:thm2}.

\section{Application: Decentralized TD(0) Learning}\label{sec:application}
This section studies the policy evaluation problem in multi-agent reinforcement learning (RL) via the decentralized {\td} algorithm \cite{doan2019finite} as an application of our tightened analysis for decentralized stochastic algorithms. 
By recognizing that \cite{doan2019finite} for multi-agent Markov Decision Process (MDP) shares many similarities with {\algoname} taking homogeneous data, we show that the existing convergence analysis for decentralized {\td} in \cite{doan2019finite} can be improved. For the mean squared error in the value function estimation, the convergence rate can be accelerated from ${\cal O}(\log t/ (\rho t))$ to ${\cal O}(1/t)$ with a transient time of \emph{zero} under the same step size rule.
\vspace{.1cm}

\noindent \textbf{Policy Evaluation.} We briefly review the setup for policy evaluation of multi-agent MDP. Consider a network of $n$ agents where agents collaborate under private local rewards and global state-action pairs. We consider the multi-agent MDP described by $(\mathcal{S}, \{\src{A}_{i}\}_{i=1}^{n}, \src{P}^{a}, \{r_{i}\}_{i=1}^{n}, \gamma)$, where  $\src{S}$, $|\src{S}| < \infty$ is a finite state space and $\src{A}_i$, $|\src{A}_i| < \infty$ is a finite action space for agent $i$. The matrix ${\cal P}^{a}\in \RR^{|{\cal S}|\times |{\cal S}|}$ is the state transition probabilities under a joint action ${a} \in {\cal A} := {\cal A}_{1} \times \cdots \times {\cal A}_{n}$. The local reward obtained by agent $i$ after taking joint action ${a}$ at state $s$ is $r_{i} (s, a)$, where $r_i : {\cal S} \times {\cal A} \to \RR_+$ is measurable.
The local reward is private and is known for agent $i$, while both system state $s$ and joint action ${a}$ are {observed} by all agents. Lastly, a policy $\pi$ is defined through the conditional probability $\pi( {a} |s)$. The reward function is defined as average of the local rewards: at state $s$, we have $R(s) \eqdef \frac{1}{n}\sum_{i=1}^{n}R_{i}(s) = \frac{1}{n} \sum_{i=1}^n \EE_{ {A} \sim \pi(\cdot |s)}[r_i( s,A )]$.

Policy evaluation is a common problem in RL. The aim is to compute the value function $V_\pi :\src{S}\rightarrow \RR$, defined to be the average discounted reward generated by $\pi$, i.e., let $\gamma \in (0,1)$ be the discount factor,
\beq \label{eq:vpis}
V_\pi (s) \eqdef \textstyle \EE_{\pi} \left[ \sum_{t=0}^{\infty} \gamma^{t+1} {R}(S_t) | S_0 = s\right],
\eeq 
where $\EE_{\pi} [\cdot]$ is the expectation taken over the MDP trajectory generated by the policy $\pi$.
Instead of evaluating \eqref{eq:vpis}, a popular formulation is to adopt linear function approximation. We consider a parametric family of linear functions $\{\phi^\top(s)\prm: \prm\in \RR^d\}$, where $\prm$ is the parameter and $\phi: {\cal S}\rightarrow \RR^d$ is a feature map.
Our goal is then to find $\prm^\star_{i}$, $i\in [n]$ such that $\phi^\top (s) \Bprm^\star \approx V_\pi(s)$ for all $s \in \src{S}$.

The decentralized {\td} algorithm in \cite{doan2019finite} is a natural extension for the {\td} algorithm \cite{sutton1988learning}. At time $t$, agent $i$ observes the tuple $\zeta_{i}^{t+1} \eqdef (S_t, A_t, S_{t}^+)$ and the local reward $r_i( S_t, A_t )$. The value function parameter is updated by
\beq \label{eq:dtd0}
    \prm_i^{t+1} = \sum_{j=1}^{n} {\bm W}_{ij}\prm_j^t + \alpha_{t+1} \, g(\prm_i^{t+1}, \zeta_{i}^{t+1}),
\eeq 
where $\alpha_{t+1} > 0$ is the step size, and
\beq 
\begin{aligned} 
& g(\prm_i^{t+1}, \zeta_{i}^{t+1}) \eqdef \phi(S_t) \delta(\prm_i^{t+1}, \zeta_{i}^{t+1}) \\
& \delta(\prm_i^{t}, \zeta_i^{t+1}) \eqdef r_i ( S_t, A_t ) + ( \gamma \phi(S_t^+) - \phi(S_{t}) )^\top \prm_i^t.
\end{aligned}
\eeq 
It is known \cite[Theorem 2]{doan2019finite} that $\EE[ \| \Bprm^t - \prm^\star \|^2 ] = {\cal O}( \log t / (\rho t ) )$ under standard conditions such as $[\phi(s)]_{s \in {\cal S}}$ is full-rank, where $\prm^\star$ is the fixed point to the projected Bellman equation. 

\subsection{Convergence Analysis}
We analyze the decentralized {\td} algorithm \eqref{eq:dtd0} under a simplified setting with \emph{independent samples}. We assume
\begin{Assumption} \label{ass:mdp-0}
For all $t \geq 0$, the tuple $\zeta^{t+1} = (S_t, A_t, S_t^+ )$ is drawn i.i.d.~following the unique stationary distribution, $\mu_\pi$, of a Markov chain induced by $\pi$ on the multi-agent MDP. In particular, one has $S_t \sim \mu_\pi$, $A_t \sim \pi( \cdot | S_t )$, $S_{t}^+ \sim {\cal P}^{A_t}( S_t, \cdot )$.
Moreover, the same tuple $\zeta^{t+1}$ is observed by all agents.
\end{Assumption}
\noindent The assumption is common for algorithms that use a replay buffer, e.g., \cite{di2022analysis}, in lieu of using a single trajectory of samples. This is called the \emph{global state model} in \cite{doan2019finite, liu2021distributed}. 

To simplify our notations, we define the quantities:
\beq \notag
\begin{aligned}
    &{\bm A}(\zeta^{t+1}) \eqdef  \phi(S_t)[ \phi^\top(S_t) - \gamma \phi^\top(S_{t}^+)],
    \\
    &{\bm b}_i(\zeta^{t+1}) \eqdef r_i ( S_t, A_t ) \phi(S_t), ~~ \bar{\bm b}(\zeta^{t+1}) \eqdef \frac{1}{n} \sum_{i=1}^{n} {\bm b}_i (\zeta^{t+1}),
    \\
    &{\bm A} \eqdef \EE_{\mu_\pi}[{\bm A}(\zeta)],
    ~~ \bar{\bm b} \eqdef \EE_{\mu_\pi}[ \bar{\bm b}(\zeta)].
\end{aligned}
\eeq 
The averaged iterate update is formulated as
\begin{align*}
    \Bprm^{t+1} = \Bprm^{t} + \alpha_{t+1} \left[ \bar{\bm b}(\zeta^{t+1}) -{\bm A}(\zeta^{t+1}) \Bprm^t
    \right].
\end{align*}
As seen from the above update, the decentralized {\td} admits a structure that is analogous to {\algoname} with $\varsigma_{H}=0$. Lastly, we assume
\begin{Assumption}\label{ass:td}
For any $i \in [n], s \in {\cal S}, a \in {\cal A}$, it holds $\norm{\phi(s)} \leq 1, |r_i(s,a)| \leq \rmax{}$. The matrix ${\bm A}$ is full rank and Hurwitz. 
\end{Assumption}
\noindent The assumption is again standard for {\td} learning \cite{bhandari2018finite,dalal2018finite,srikant2019finite}.
As a consequence of A\ref{ass:td}, the matrix ${\bm A}$ is positive definite and we define $\lambdamin, \lambdamax$ as the minimum, maximum eigenvalue of $\frac{ {\bm A} + {\bm A}^\top }{2}$, respectively. We also observe that:
\beqq 
\begin{aligned}
    & \norm{ {\bm A}(\zeta) } \leq 1+\gamma =: \beta,~\forall~\zeta \in {\cal S} \times {\cal A} \times {\cal S},
    \\
    & \EE_{\mu_\pi} [\norm{ \bar{\bm b} - \bar{\bm b}(\zeta) }^2] \leq \frac{2 \rmax{2}}{n} =: \frac{\sigma^2}{n}.
\end{aligned}
\eeqq 
We proceed by observing the descent lemma for \eqref{eq:dtd0}:
\begin{lemma}\label{lem:td_des}
    Under A\ref{ass:mdp-0}, A\ref{ass:td}, if the step size $\alpha_{t}$ satisfies $\sup_{t \geq 1} \alpha_t \leq \lambdamin / \beta^2$, then, it holds for any $t\geq 1$ that,
    \begin{align*}
    \EE_{t} \| \Bprm^{t+1} - \prm^\star \|^2 & \leq \left( 1 - \alpha_{t+1} \lambdamin  \right) \| \Bprm^t-\prm^\star \|^2 \\
    & + \alpha_{t+1}^2 \left( \frac{ \sigma^2 }{n} + 4 \beta^2 \| \prm^\star \|^2 \right) 
    \end{align*}
\end{lemma}
\noindent The proof can be found in Appendix~\ref{app:td0}.
Solving the recursion in the above lemma yields:
\begin{theorem}\label{thm-td} 
    Under A\ref{ass:mdp-0}, A\ref{ass:td}. Suppose that the step size satisfies,
    \beqq 
        \frac{\alpha_t}{\alpha_{t+1}} \leq \sqrt{1+ (\lambdamin/2) \alpha_{t+1}} ,~~\alpha_t \leq \frac{ \lambdamin }{ \beta^2 },~\forall~t \geq 1.
    \eeqq 
    Then, for any $t\geq 1$, it holds that
    \begin{align*}
        \EE\|\Bprm^{t+1} - \prm^\star \|^2 & \textstyle \leq \prod_{i=1}^{t+1} \left(1-{\lambdamin}\alpha_{i}\right) \| \Bprm^{0} - \prm^\star \|^2  
        \\
        &\quad + \frac{2}{\lambdamin} \left( \frac{ \sigma^2 }{n} + 4 \beta^2 \| \prm^\star \|^2 \right) \alpha_{t+1}.
    \end{align*}
\end{theorem}
\noindent
Theorem~\ref{thm-td} is a direct consequence of Lemma~\ref{lem:td_des} and \ref{lem:aux}. 
If we let $\alpha_{t}={\cal O}(1/t)$, then convergence rate of decentralized {\td} is ${\cal O}(1/t)$. We notice that the transient time is \emph{zero} since the algorithm employs the \emph{same} ${\bm A}(\zeta)$ across all the agents. We remark that a related observation has been made on decentralized {\td} algorithm in \cite{liu2021distributed}.  

\section{Numerical Simulations}\label{sec:simulation}
This section presents simulation examples to validate our theoretical findings. 

\subsection{Quadratic Minimization} 
The first example considers a quadratic stochastic optimization problem with
\beq\label{eq:quad}
    \min_{\prm\in\RR^{d}}~f(\prm) = \frac{1}{n}\sum_{i=1}^{n} \EE_{ Z_i \sim {\sf B}_i} \left[ \frac{1}{2}\prm^\top ( \tilde{\bm A}_{i} + \tilde{\bm A}_{i}^\top )\prm + \prm^\top \tilde{\bm b}_i\right],
\eeq
where $Z_i \equiv (\tilde{\bm A}_{i}, \tilde{\bm b}_i)\in \RR^{d\times d}\times \RR^d$ is a sample drawn from $\textbf{\sf B}_{i}$ accessible by agent $i$, satisfying $\EE[\tilde{{\bm A}}_i] = {\bm A}_{i}$, $\EE[\tilde{\bm b}_{i}] = {\bm b}_{i}$. 
Assuming that $\sum_{i=1}^{n} ( {\bm A}_{i} + {\bm A}_{i}^\top )$ is positive definite, the optimal solution of (\ref{eq:quad}) admits the closed form
$\prm^{\star} =  - \left(\sum_{i=1}^{n} \left({\bm A}_{i} + {\bm A}_{i}^\top \right) \right)^{-1}\sum_{i=1}^{n} {\bm b}_{i}$.

To simulate a \emph{heterogeneous data} setting, 
we let ${\sf B}_i$ be an empirical distribution given by the dataset ${\sf B}_i = \{ (\tilde{\bm A}_i^1, \tilde{\bm b}_i^1), \ldots, ( \tilde{\bm A}_i^{|{\sf B}_i|}, \tilde{\bm b}_i^{|{\sf B}_i|} ) \}$ where we set $|{\sf B}_i| = 500$.
For the $j$th sample, each element of $\tilde{\bm b}_i^j$ is i.i.d.~generated with ${\cal N}(0,1)$; and each element of $\tilde{\bm A}_i^j$ follows the distribution:
\[ 
[\tilde{\bm A}_i^j]_{k, \ell} 
\sim \begin{cases}
{\cal N}(0,1), ~~ \text{if $k \neq \ell$ or $k \notin \Pi_{2,i}^d$},  \\
{\cal N}(0,2), ~~ \text{if $k = \ell \in \Pi_{2,i}^{d}$} , \\
\end{cases}
\]
where $\Pi_{2,i}^{d} \subseteq \{1,\ldots,d\}$ with $|\Pi_{2,i}^{d}| = 2$ is an agent-specific subset of coordinates. Each $\Pi_{2,i}^{d}$ is generated by taking a distinct 2-combination of $\{1,\ldots,d\}$. The \emph{homogeneous data} setting is set by extending the above. The samples $Z_i$ drawn by agent $i$ follows the distribution ${\sf B} \equiv \frac{1}{n} \sum_{i=1}^n {\sf B}_i$, i.e., $Z \sim {\sf B}$ is uniformly drawn from $\{{\sf B}_{i}\}_{i=1}^{n}$. Moreover, we also set ${\sf B}_i(\alpha) \eqdef \alpha {\sf B}+(1-\alpha) {\sf B}_i$ for \emph{partially heterogeneous data}, i.e., $Z \sim {\sf B}_i(\alpha)$ is drawn with probability $\alpha$ from ${\sf B}$, $1-\alpha$ from ${\sf B}_i$. The distribution becomes more heterogeneous as $\alpha \to 0$.  

\begin{figure}[t]
    \centering
    \includegraphics[width=.49\linewidth]{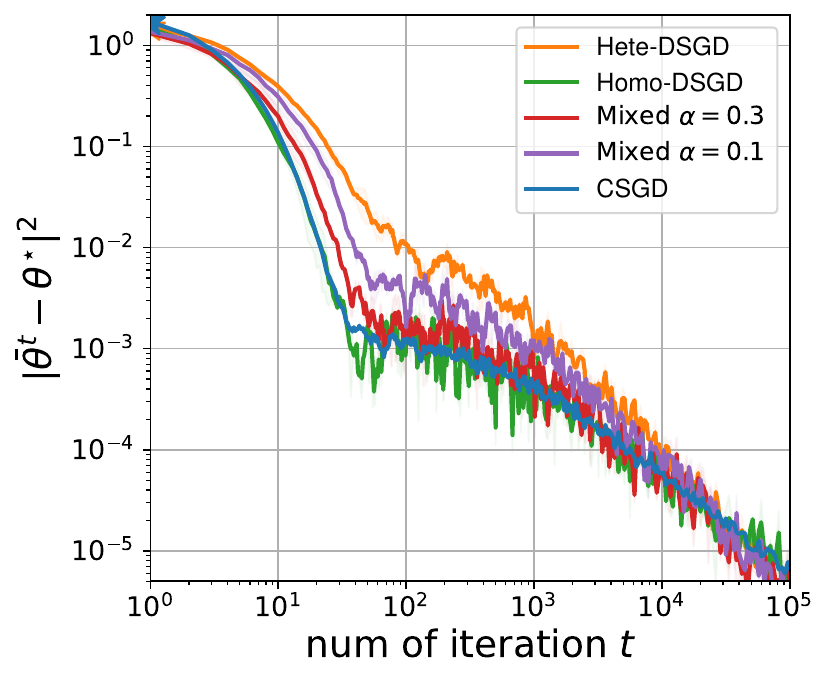}
    \includegraphics[width=.49\linewidth]{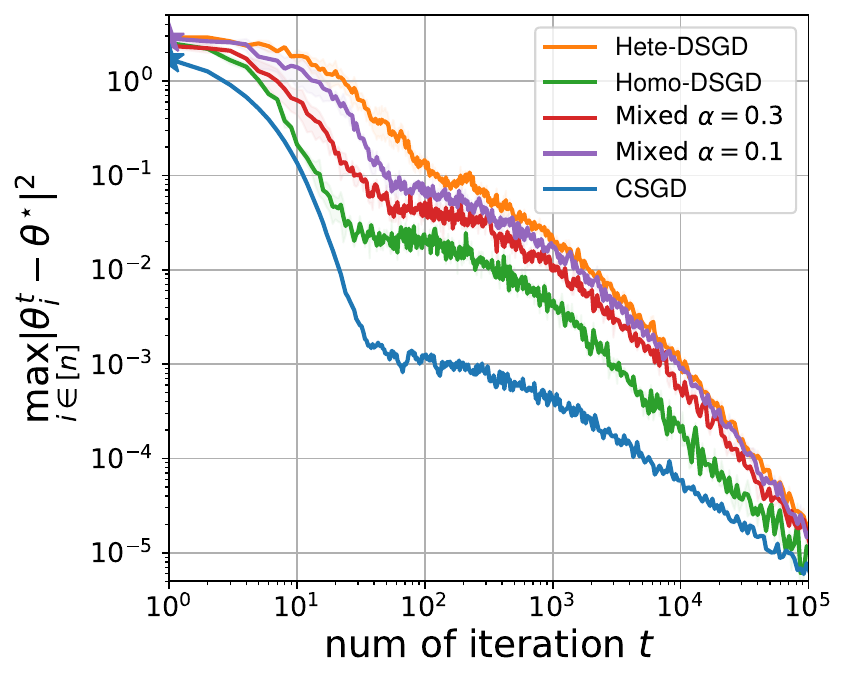}
    \caption{\textbf{Quadratic Minimization} Comparison of {\sf Homo-DSGD}, {\sf Hete-DSGD}, mixed distribution ${\sf B}(\alpha)$ ($\alpha=0.1, 0.3$) alongside with {\sf CSGD} on convergence behavior under average optimality gap measurement (\emph{First Figure}) and worst optimality gap (\emph{Second Figure}). 
     }
    \label{fig:quad}
\end{figure}

We let $d=10$, $n=20$ and agents are connected by a ring graph with mixing matrix ${\bm W}$ where ${\bm W}_{ii} = 0.3$, ${\bm W}_{i, i-1}={\bm W}_{i, i-1} = 0.35$. 
Note that we have $\rho=0.034$ in this case. For {\algoname}, each agent draw one sample to generate the stochastic gradient estimate. For brevity, we will refer to {\algoname} under the prescribed homogeneous data setting as {\sf Homo-DSGD}, heterogeneous data setting as {\sf Hete-DSGD}, and partially heterogeneous data setting as {\sf Mixed $\alpha$} with $\alpha \in \{ 0.1,0.3 \}$. As a benchmark, we consider {\sf CSGD} which draws a mini-batch of $n$ samples from ${\sf B}$. Unless otherwise specified, the stepsizes are set as $\gamma_{k} = 10/(500+k)$. 
 
In Fig. \ref{fig:quad}, we plot the average/worst optimality gap $\normtxt{\Bprm^t-\prm^{\star}}^2$ against iteration number $t$ for the test algorithms over $20$ repeated runs. 
Observe that the {\sf DSGD} algorithms under the three settings approach the same steady state convergence behavior as the centralized algorithm {\sf CSGD} as $t\rightarrow \infty$ validating with Corollary \ref{cor:2}. Observe that {\sf Homo-DSGD} performs almost same as {\sf CSGD} with the shortest transient time, while {\sf Hete-DSGD} needs more time ($\approx 3\times 10^{4}$) to catch up with the performance of {\sf CSGD}. The transient time of {\sf DSGD} is improved to around $10^{4}$ iterations when the degree of heterogeneity is decreased (with $\alpha=0.3$). The observation corroborates with Theorem \ref{thm2}. Besides, the acceleration of transient degrades for the worst-case optimality gap $\max_{i\in [n]}\norm{\prm_i^t-\prm^\star}^2$ and the transient time of {\sf Homo-DSGD} is increased.
\begin{figure}
    \centering
    \includegraphics[width=.8\linewidth]{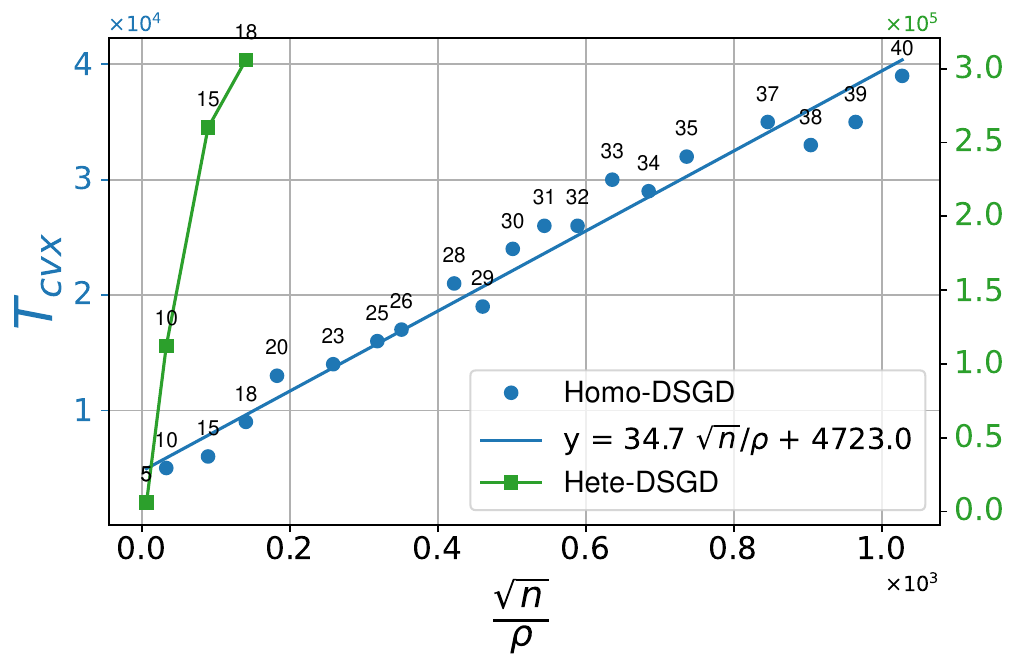}
    \caption{Verifying the ${\cal O}(\sqrt{n}/\rho)$ transient time against network size $n$. Notice that the transient time for {\sf Homo-SGD} and {\sf Hete-SGD} are plotted with different scale in the $y$-axis for better illustration. The number above the marker denotes the network size $n$.
    }
    \label{fig:tran}
\end{figure}

Lastly, we verify the transient time bound ${\cal O}(\sqrt{n}/\rho)$ in Corollary \ref{cor:cvx_tran} for {\sf Homo-DSGD}. The transient time of decentralized algorithms are estimated as 
\[
    \textstyle \widehat{T}_{\sf cvx} = {\displaystyle \inf_{t\geq 1}} \left\{ t: \frac{\max_{j\in[n]}\normtxt{\Bprm^{t-i}_{j}-\prm^\star}^2 }{\normtxt{\Bprm^{t}_{\tt Cen}-\prm^\star}^2} \leq \frac{1}{4},~i=1,\cdots, 3000 \right\},
\]
where $\Bprm^{t}_{\tt Cen}$ is the running average of {\sf CSGD} under the same settings as {\algoname}.
Fig. \ref{fig:tran} shows the dependence of $\widehat{T}_{\sf cvx}$ and $\frac{\sqrt{n}}{\rho}$ under different data settings. We see that the estimated transient time of {\sf Homo-DSGD} has a linear dependence on $\frac{\sqrt{n}}{\rho}$ which corroborates with our theory. 

\begin{figure*}[t]
    \centering
    \includegraphics[width=.98\linewidth]{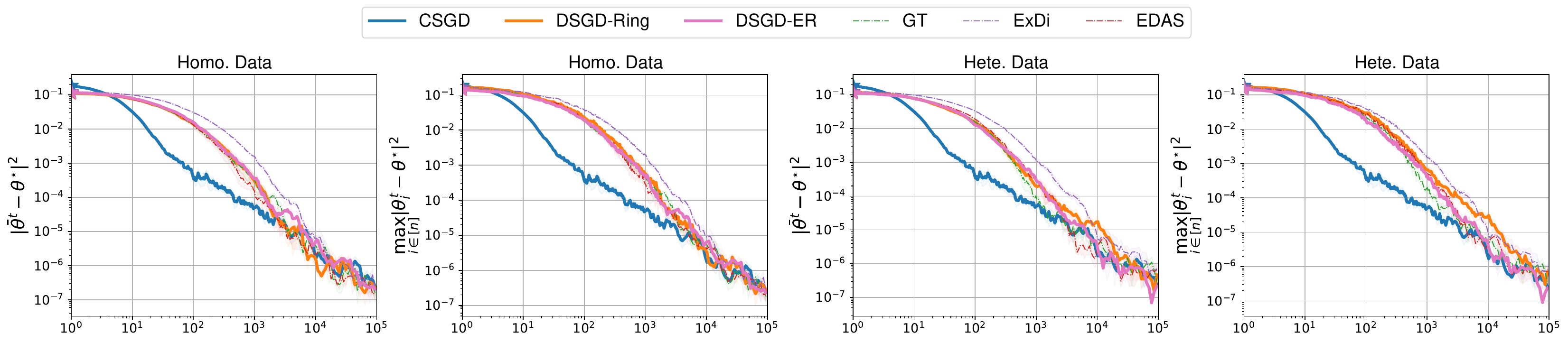}
    \vspace{-.1cm}
    \caption{\textbf{Logistic Regression Example} Comparison between CSGD and decentralized algorithms on homogeneous data (First \& Second Figure)  and heterogeneous data  (Third \& Fourth Figure), measured with t he average optimality gap measurement and worst optimality gap.
    {The spectral gap of the ER graph is \update{$\rho=0.057$}, and for the ring graph it is \update{$\rho=0.0066$}. Unless otherwise specified, the above simulations are conducted on the ring graph.}\vspace{-.3cm}}
    \label{fig:logistic}
\end{figure*}

\subsection{Logistic Regression}

We consider a binary classification problem through training a support vector machine (SVM) across 30-agents with synthetic datasets. The ground truth is given by $[\prm_{\text{o}}, b_{\text{o}}]\in {\cal U}[-1, -1]^{5+1}$. For each $i$, the data distribution ${\sf B}_i$ is taken to be the empirical distribution of $200$ samples $\{ x_j^i , y_j^i \}_{j=1}^{200}$, which are generated as $x_j^i \sim {\cal U}[-1,1]^{5}$, $y_j^i = (\sign(\Pscal{x_j^i}{\prm_{\text{o}}}+b_{\text{o}})+1)/2.$ Denote ${\sf B}$ as the empirical distribution of $\{ \{ x_j^i , y_j^i \}_{j=1}^{200} \}_{i=1}^{30}$ with size $m=6000$. The classifier can be obtained by solving (\ref{eq:opt}) using logistic loss
\beq\label{eq:logistic}
    \ell( \prm; z ) = \frac{r}{2}\norm{\prm}^2 + \log({1 + \exp( y \pscal{x}{\prm}) }) - y\Pscal{\prm}{x},
\eeq
where $z\equiv(x,y)\in \RR^d \times \{0,1\}$ is feature-label pair. With $r=1$, $\ell(\prm; z)$ is $r$-strongly convex and $L$-smooth. We let {$\gamma_{t}=5/(100 + t)$}, $\forall t$. For decentralized algorithm, the mini-batch size is set as $4$, and {\sf CSGD} uses a batch size of $4n$. {In addition to the ring graph topology, we simulate the performance of {\sf DSGD} on an Erdos-Renyi (ER) graph generated with connectivity of $p=0.05$. The mixing matrix weights for the ER graph are computed by the Metropolis-Hasting rule while ensuring a self-weight of ${\bm W}_{i,i} \approx 0.5$.}

Fig.~\ref{fig:logistic} compares the performance of {\sf DSGD}, Gradient tracking ({\sf GT}) \cite{pu2021distributed}, Exact diffusion ({\sf ExDi}) \cite{yuan2023removing} under the homogeneous/heterogeneous data settings. 
With homogeneous data, {\sf DSGD} can achieve as good performance as other sophisticated algorithms and achieve comparable performance with its centralized counterpart. On the other hand, without data homogeneity, {\sf DSGD} requires more iterations to overcome the influence of decentralization. {Furthermore, with heterogeneous data, the transient time of {\sf DSGD} is reduced with the ER topology which has a larger spectral gap than the ring topology, suggesting that the transient time in heterogeneous data setting is sensitive to the spectral gap.} These results coincide with our theoretical analysis.

{
\begin{figure}[t]
    \centering
    \begin{tikzpicture}[scale=0.7] 
        \draw[step=1cm,color=black] (0,0) grid (4,4);
        
        \foreach \x in {0,1,2,3}
            \foreach \y in {0,1,2,3}
                \node at (\x+0.5,\y+0.5) {\pgfmathparse{int(4*\y+\x+1)}\pgfmathresult};
    \end{tikzpicture}
    \captionof{figure}{State space of grid world.\vspace{-.2cm}}\label{fig:grid} 
\end{figure}

\subsection{Decentralized TD(0) Learning}
We consider a similar setting to \cite{liu2021distributed} on the {\tt GridWorld} of a $4\times 4$ grid as shown in Fig.~\ref{fig:grid}), where ${\cal S} = \{1,2\cdots, 16\}$ and ${\cal A}=\{\text{left, right, up, down}\}$. If the action leads out of the grid, then the next state will remain to be the current state. The discounting factor in the MDP is $\gamma=0.9$. For each $s,a$, we generate the reward table by $r(s,a) \sim {\cal N}(1,10)$. We consider evaluating the \emph{random policy}, i.e., the agent chooses one of the 4 actions with equal probability. The feature vectors are generated as $\phi(s) = (1,0,0,0)^\top$ for $s\in \{1,2,5,6\}$, $\phi(s) = (0,1,0,0)^\top$ for $s\in \{3,4,7,8\}$ $\phi(s) = (0,0,1,0)^\top$ for $s\in \{9,10,13,14\}$ and $\phi(s) = (0,0,0,1)^\top$ for $s\in \{11,12,15,16\}$. Furthermore, the samples for {\td} algorithm \eqref{eq:dtd0} are generated according to A\ref{ass:mdp-0}.

We consider a multi-agent scenario where $n=10$ agents connected by a ring graph collaboratively evaluate the random policy $\pi$. The mixing matrix ${\bm W}$ is set as ${\bm W}_{i,j} = 0.8$ if $i=j$, ${\bm W}_{i,j} = 0.1$ if $|i-j| = 1$, and ${\bm W}_{i,j} = 0$ otherwise. The optimal $\prm^\star$ is obtained by solving the Bellman Equation ${\bm A} \prm^\star = \bar{\bm b}$. 
In Fig. \ref{fig:td}, we plot optimality gap against iteration number. As observed, both centralized and decentralized {\td} converges to $\prm^\star$ at the rate of ${\cal O}(1/t)$. Moreover, there is no observable transient time exhibited by the decentralized algorithm. Our finding corroborates the conclusions in Theorem \ref{thm-td}.

\begin{figure}[t]
\centering
    \includegraphics[width=.49\linewidth]{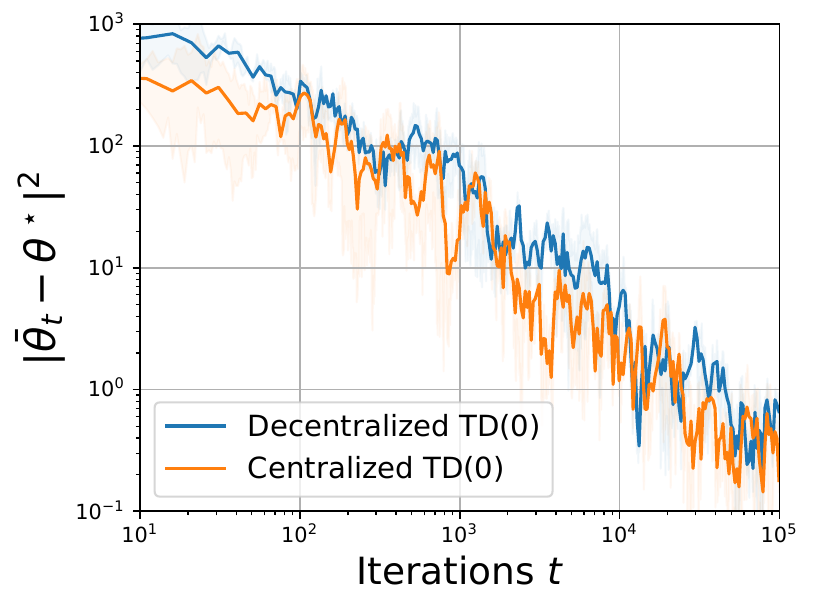}
    \includegraphics[width=.49\linewidth]{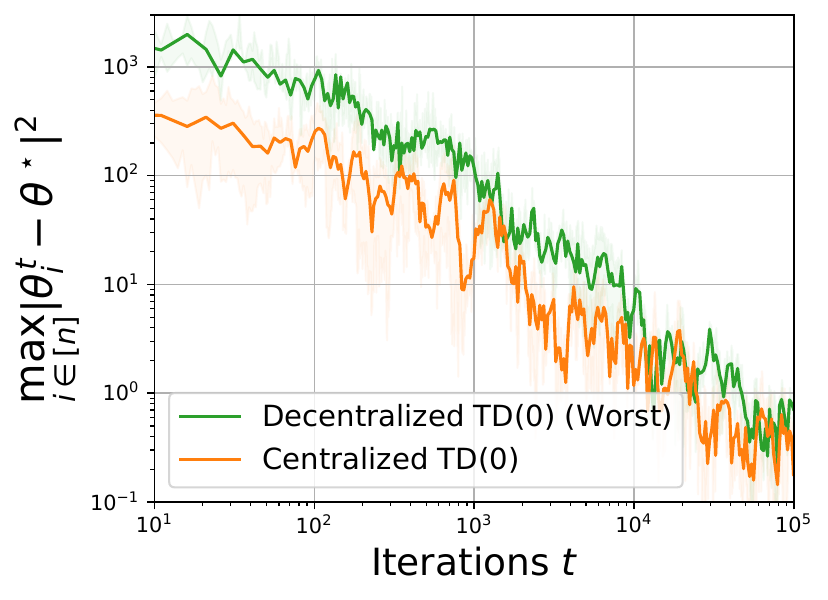}
    \caption{{\bf {\td} Learning} Comparison of {\td} algorithms with linear function approximation. (Left) Error of averaged model, (Right) Error of worst model in network.}\vspace{-.3cm}
    \label{fig:td}
\end{figure}}

\section{Conclusion}
We provide a fine grained analysis for the convergence rate of {\sf DSGD}. We focus on the role of data homogeneity when the loss function is smooth (possibly non-convex), and strongly convex. Particularly, the plain {\sf DSGD} algorithm is shown theoretically and empirically to achieve fast convergence rate when the data distribution across agents are \emph{similar} to each other. 
\update{Our findings demonstrate that the {\algoname} algorithm, despite being simple, can already achieve fast convergence that has on-par performance with sophisticated algorithms such as gradient tracking.}
{For future work, a promising direction is to exploit the high order smoothness (cf.~A\ref{ass:Hsmooth}) property to obtain tight convergence rates for other algorithms such as \cite{xin2021improved, huang2022improving, yuan2023removing}. We also note the recent works \cite{cutkosky2020momentum,fang2019sharp} which exploited A\ref{ass:Hsmooth} for accelerating stochastic gradient methods in a centralized setting.}
\update{We also envision that more efficient algorithms than {\algoname} can be developed that better adapts to data homoegeneity, e.g., \cite{tian2022acceleration, le2023refined}.}

\appendices
\section{Unified Consensus Error Bound} \label{app:unif-cons-bd}
Before we present the the proof of lemmas, we denote the following notations throughout the appendix: 
\beqq 
\begin{aligned} 
& \grd f_i^t := \grd f_i(\prm_i^t),~\Tgrd f_i^t := \grd \ell( \prm_i^t ; Z_i^{t+1} ), \\
& \Prm^t \eqdef \big( \prm_1^t, \ldots, \prm_n^t \big)^\top,~\Tgrd F^t \eqdef \left(\Tgrd f_1^t, \ldots, \Tgrd f_n^t\right)^\top.
\end{aligned}
\eeqq 
Note that $\Prm^t$, $\Tgrd F^t$ are $n \times d$ matrices.

In this section, we first present a unified consensus error bound that subsumes Lemma \ref{lem:ceb-ncvx} and Lemma \ref{lem:ceb-scvx}, as follows:
\begin{lemma}
    Under A\ref{ass:smooth}, \ref{ass: graph}, \ref{ass:SecOrdMom}, if
    $\sup_{t\geq 1}\gamma_{t}\leq \frac{\rho}{2 \sqrt{2(\sigma_{1}^2+2L^2)}}$, then for any $t\geq 0$, it holds that
    \begin{align}\label{eq:eee}
        & \EE_t [\norm{\CSE{t+1}}_F^2] \leq \left(1- \frac{\rho}{2} \right)  \norm{\CSE{t}}_F^2 
        \\
        & +  \frac{2n  \gamma_{t+1}^2}{\rho} \left[ (\varsigma^2 + \sigma_0^2 ) + 2\sigma_1^2 \normtxt{\Bprm^t-\prm^\star}^2 \right].
        \nonumber
    \end{align}
If in addition A\ref{ass:SecOrdMom2} holds, $\sup_{t \geq 1}\gamma_{t}\leq \frac{\rho}{\sqrt[4]{c_3}}$, 
Then for any $t \geq 0$, it holds
\begin{align}\label{eq:ff}
    &\EE_t \norm{\CSE{t+1}}^4_F 
    \leq (1-\frac{\rho}{2})\norm{\CSE{t}}_F^4 
    \\
    &+ 54 n^2  \left[(\bar{\sigma}_{0}^4 + 4\varsigma^4 )+8\bar{\sigma}_{1}^4\normtxt{\Bprm^t-\prm^\star}^4 \right]\frac{\gamma_{t+1}^4}{\rho^3},
    \nonumber
\end{align}
where $c_3 \eqdef 864n (\bar{\sigma}_{1}^4 + 8 L^4)$.
\end{lemma}
\begin{proof}

By noting that ${\bm I}-\frac{1}{n}{\bm 1}{\bm 1}^\top = {\bm U} {\bm U}^\top$, we observe the following chain:
\begin{align}
    & \Prm^{t+1}-\BPrm^{t+1} = {\bm U}{\bm U}^\top \left({\bm W}\Prm^t -\gamma_{t+1}\Tgrd F^t\right) \nonumber \\
    &= {\bm U}{\bm U}^\top \left[{\bm W} \left({\bm U \bm U}^\top + \frac{1}{n}{\bm 1 \bm 1}^\top\right)\Prm^t -\gamma_{t+1}\Tgrd F^t\right] \nonumber \\
    &= {\bm U}\left({\bm U}^\top {\bm W}{\bm U} \right) {\bm U}^\top \Prm^t - \gamma_{t+1}{\bm U}{\bm U}^\top \Tgrd F^t.\label{eq:consensuserr}
\end{align}
Using A\ref{ass: graph} shows that $\norm{\CSE{t+1} }_F^2$ is upper bounded by
\begin{align*}
    & (1+\alpha) (1-\rho)^2 \norm{ \CSE{t} }_F^2 \!+\! \left(1+\frac{1}{\alpha}\right) \gamma_{t+1}^2\norm{{\bm U \bm U}^\top \Tgrd F^t}_F^2,
\end{align*}
for any $\alpha > 0$. Setting $\alpha=\frac{\rho}{1-\rho}$ gives
\begin{align}\label{eq:cseb}
    \norm{\CSE{t+1}}_F^2 
    & \leq (1-\rho)\norm{\CSE{t}}_F^2+\frac{\gamma_{t+1}^2}{\rho}  \norm{{\bm U \bm U}^\top \Tgrd F^t}_F^2.
\end{align}
Next, we aim to bound the last term:
\begin{align*}\textstyle
    &\quad \EE_{t}\norm{{\bm U \bm U}^\top \Tgrd F^t}_F^2 = \EE_{t} \bigg[ \sum_{i=1}^{n} \bigg\| \Tgrd f_i^t-\frac{1}{n}\sum_{j=1}^{n}\Tgrd f_j^t \bigg\|^2 \bigg]
    \\
    & = \EE_{t} \Bigg[\sum_{i=1}^{n} \Big\|\left(\Tgrd f_i^t - \grd f_i^t \right) + \bigg(\grd f_i^t -\frac{1}{n}\sum_{j=1}^{n}\grd f_j^t\bigg)
    \nonumber
        \\
        &\quad - \frac{1}{n}\sum_{j=1}^{n}(\Tgrd f_j^t-\grd f_j^t)\Big\|^2\Bigg] \nonumber
    \\
    &\overset{(a)}{\leq} \sum_{i=1}^{n} \EE_{t} \left[ \norm{\Tgrd f_i^t - \grd f_i^t}^2\right] + \sum_{i=1}^{n}\bigg\| \grd f_i^t - \frac{1}{n}\sum_{j=1}^{n}\grd f_j^t \bigg\|^2 
        \\
        &\quad + \frac{1}{n^2}\sum_{j=1}^{n} \EE_{t} \left[ \big\| \Tgrd f_j^t-\grd f_j^t \big\| ^2\right] \nonumber
    \\
    &\leq  2\sum_{i=1}^{n} \EE_{t} \left[ \norm{\Tgrd f_i^t - \grd f_i^t}^2 \right] + \sum_{i=1}^{n} \bigg\| \grd f_i^t-\frac{1}{n}\sum_{j=1}^{n}\grd f_j^t \bigg\|^2, \nonumber
\end{align*}
where inequality (a) is due to $\EE[ \Tgrd f_j^t - \grd f_j^t ] = 0$. For the first term, applying A\ref{ass:SecOrdMom} yields 
\begin{align}\label{eq2a}
    &\sum_{i=1}^{n}  \EE_{t} \left[ \norm{\Tgrd f_i^t - \grd f_i^t}^2 \right]
    \leq  (n \sigma_0^2+ \sigma_1^2 \sum_{i=1}^{n} \norm{\prm_i^t-\prm^\star}^2 )  \nonumber
    \\
    &\leq \left(n\sigma_0^2+2\sigma_1^2\norm{\CSE{t}}^2_F +2n\sigma_1^2 \normtxt{\Bprm^t-\prm^\star}^2\right).
\end{align}
For the second term, one has the following upper bound
\small \begin{align}\label{eq3}
    &\sum_{i=1}^{n} \bigg\| \grd f_i^t-\frac{1}{n}\sum_{j=1}^{n}\grd f_j^t \bigg\|^2 
    \\
    &\leq 2\sum_{i=1}^{n} \left( \norm{\grd f_i^t -\grd f(\prm_i^t)}^2 + \bigg\| \frac{1}{n}\sum_{j=1}^{n}(\grd f_j^t-\grd f_j(\prm_i^t))\bigg\|^2\right) \nonumber
    \\
    &\leq 2\sum_{i=1}^{n} \left( \norm{\grd f_i^t -\grd f(\prm_i^t)}^2 + \frac{1}{n}\sum_{j=1}^{n}\bigg\|\grd f_j^t-\grd f_j(\prm_i^t) \bigg\| ^2\right)
    \nonumber
    \\
    &\overset{(a)}{\leq} 2n\varsigma^2 \!+\! \frac{2L^2}{n}\sum_{i=1}^{n}\sum_{j=1}^{n}\norm{\prm_j^t-\prm_i^t}^2 
    \leq 2n\varsigma^2 + {4L^2}\norm{\CSE{t}}^2_F,
    \nonumber
\end{align}
\normalsize where we apply A\ref{ass:smooth} on the two terms respectively to obtain inequality (a). \noindent
Combining inequality (\ref{eq2a}) and (\ref{eq3}) derives 
\begin{eqnarray}\label{eq4}
   \begin{aligned}
    \EE_t \left[ \norm{{\bm U \bm U}^\top \Tgrd F^t}_F^2 \right] &\leq 2n(\sigma_0^2 +\varsigma^2) + 4n\sigma_{1}^2\norm{\Bprm^t - \prm^\star}^2
    \\
    &\quad + 4(\sigma_{1}^2+2{L^2})\norm{\CSE{t}}^2_F.
\end{aligned} 
\end{eqnarray}
Substituting above inequality (\ref{eq4}) back to (\ref{eq:cseb}) and requiring $\sup_{t\geq 1}\gamma_{t}\leq \frac{\rho}{2 \sqrt{2(\sigma_1^2+ 2 L^2)}}$, we have the bound (\ref{eq:eee}).

Similar to the proof of \eqref{eq:eee}, using A\ref{ass: graph} and \eqref{eq:consensuserr}, the following holds for any $\alpha, \beta> 0$,
\begin{align*}
\norm{\CSE{t+1}}^4_F & \leq  (1+\beta)(1+\alpha)^2  (1-\rho)^4 \, \norm{\CSE{t}}_F^4 \\
&\quad \textstyle + (1+\frac{1}{\beta})(1+\frac{1}{\alpha})^2  \gamma_{t+1}^{4} \norm{{\bm U}{\bm U}^\top\Tgrd F^t}^4_F .
\end{align*}
Setting $\alpha =\beta = \frac{\rho}{1-\rho}$ leads to the following
\begin{align}\label{eq:bb}
\norm{\CSE{t+1}}^4_F \leq (1-\rho)\norm{\CSE{t}}_F^4 + \frac{\gamma_{t+1}^4}{\rho^3}\norm{{\bm U}{\bm U}^\top\Tgrd F^t}_F^4.
\end{align}
We now bound the last term as:
\small \begin{align*}
    & \textstyle \frac{1}{n} \norm{{\bm U \bm U}^\top \Tgrd F^t}_F^4 = \frac{1}{n} \left(\sum_{i=1}^{n} \norm{\Tgrd f_i^t-\frac{1}{n}\sum_{j=1}^{n}\Tgrd f_j^t}^2 \right)^2
    \\
    & \textstyle \leq \sum_{i=1}^{n} \norm{\Tgrd f_i^t-\frac{1}{n}\sum_{j=1}^{n}\Tgrd f_j^t}^4
    \\
    & \!=\! \sum_{i=1}^{n} \norm{\Tgrd f_i^t\!-\!\grd f_i^t \!+\! \grd f_i^t \!-\!\frac{1}{n}\sum_{j=1}^{n}\grd f_j^t\!-\!\frac{1}{n}\sum_{j=1}^{n}\left(\Tgrd f_j^t\!-\!\grd f_j^t\right)}^4.
\end{align*}
\normalsize Due to the fact that $(a+b+c)^4\leq 27(a^4+b^4+c^4)$, we have
\begin{align}\label{eq-a}
    &\EE_t \left[ \norm{{\bm U \bm U}^\top \Tgrd F^t}_F^4 \right] 
    \leq 27n\sum_{i=1}^{n} \Bigg[ \EE_t \left[ \normtxt{\Tgrd f_i^t - \grd f_i^t}^4 \right] \\
    &+ \Big\|\grd f_i^t - \frac{1}{n}\sum_{j=1}^{n}\grd f_j^t\Big\|^4 + \EE_t \Big[ \Big\|{\frac{1}{n}\sum_{j=1}^{n}(\Tgrd f_j^t-\grd f_j^t)}\Big\|^4 \Big] \Bigg].
    \nonumber
\end{align}
For the last term, applying Cauchy-Schwarz inequality gives
\begin{align*} 
    \bigg\| \frac{1}{n}\sum_{j=1}^{n}(\Tgrd f_j^t\!-\!\grd f_j^t) \bigg\|^4 
    &\leq \frac{1}{n} \sum_{j=1}^{n} \norm{\Tgrd f_j^t-\grd f_j^t}^4.
\end{align*}
Substituting back to inequality (\ref{eq-a}) leads to 
\beqq 
\begin{aligned}
    \EE_t \left[ \norm{{\bm U \bm U}^\top \Tgrd F^t}_F^4 \right] 
    &\leq 27 n \Big[2 \, {\sum_{i=1}^{n} \EE_t \Big[ \norm{\Tgrd f_i^t - \grd f_i^t}^4 \Big] } 
    \\
    &\quad + \sum_{i=1}^{n}\Big\|{\grd f_i^t - \frac{1}{n}\sum_{j=1}^{n}\grd f_j^t}\Big\|^4 \Big].
\end{aligned}
\eeqq
The first term in the above can be controlled with A\ref{ass:SecOrdMom2}:
\begin{align}\label{eqbb}
    & \sum_{i=1}^{n} \EE_t \Big[\norm{\Tgrd f_i^t - \grd f_i^t}^4 \Big] \leq n \bar{\sigma}_{0}^4 +\bar{\sigma}_{1}^4 \sum_{i=1}^{n}\norm{\prm_i^t-\prm^\star}^4 \nonumber
    \\
    &\leq n \bar{\sigma}_{0}^4+ {8} \bar{\sigma}_{1}^{4}\norm{\CSE{t}}^4_F +8n \bar{\sigma}_{1}^{4} \norm{\Bprm^t-\prm^\star}^4 . 
\end{align}
For the second term, we get
\small \begin{align*}
    &\sum_{i=1}^{n}\Big\|\grd f_i^t-\frac{1}{n}\sum_{j=1}^{n}\grd f_j^t\Big\|^4 
    \\
    &\leq 8\sum_{i=1}^{n} \left( \norm{\grd f_i^t -\grd f(\prm_i^t)}^4 + \norm{\frac{1}{n}\sum_{j=1}^{n}(\grd f_j^t-\grd f_j(\prm_i^t))}^4\right)
    \\
    &\leq 8\sum_{i=1}^{n} \left( \norm{\grd f_i^t -\grd f(\prm_i^t)}^4 + \frac{1}{n}\sum_{j=1}^{n}\norm{\grd f_j^t-\grd f_j(\prm_i^t)}^4\right).
\end{align*}
\normalsize Applying A\ref{ass:smooth}, we can derive
\begin{align}\label{eqaa}
    {\sum_{i=1}^{n}\Big\|{\grd f_i^t-\frac{1}{n}\sum_{j=1}^{n}\grd f_j^t}\Big\|^4} \leq 8n\varsigma^4 + 128L^4\norm{\CSE{t}}^4_F.
\end{align}
Combining (\ref{eqbb}) and (\ref{eqaa}), we have 
\begin{align*}
    & \EE_t \left[ \normtxt{{\bm U \bm U}^\top \Tgrd F^t}_F^4 \right] 
    \leq 54n^2(\bar{\sigma}_{0}^4+4\varsigma^4)  
    \\
    & + 432 n(\bar{\sigma}_{1}^4 + 8L^4)\norm{\CSE{t}}^4_F + 432n^2 \bar{\sigma}_{1}^4 \normtxt{\Bprm^t-\prm^\star}^4.
\end{align*}
Combining with (\ref{eq:bb}) and the step size condition $\sup_{t\geq 1}\gamma_{t}^4 \leq \frac{\rho^4}{864n(\bar{\sigma}_{1}^4 + 8 L^4)}$, we have \eqref{eq:ff} and the proof is concluded.
\end{proof}

\section{Missing Proofs for Smooth Case}
\subsection{Proof of Corollary \ref{cor:1}}\label{app:ncvx1}
We first show the following descent lemma:
\begin{lemma}\label{lem:aa}
Under A\ref{ass:smooth},\ref{ass: graph},\ref{ass:SecOrdMom}[$\sigma_{0}=\sigma$, $\sigma_{1}=0$], if $\sup_{t \geq 1} \gamma_{t} \leq\frac{1}{4L}$, then for any $t \geq 0$, it holds
    \begin{align} \label{eq:concl-aa}
        \EE_t[ f( \Bprm^{t+1} ) ] &\leq  f( \Bprm^t ) - \frac{\gamma_{t+1}}{4} \| \grd f( \Bprm^t ) \|^2 + \frac{ \gamma_{t+1}^2 L \sigma^2}{2n} \\
        &+ {2}\gamma_{t+1} \frac{L^2}{n} \norm{ \CSE{t} }^2_F. \nonumber
    \end{align}
\end{lemma}
\begin{proof}

Using A\ref{ass:smooth}-\ref{item:lips} and the update rule \eqref{eq:avg_it}, we have{\small 
\begin{align*}
f( \Bprm^{t+1} ) \leq  f( \Bprm^t )  - \langle \grd f( \Bprm^t ) | {\textstyle \frac{ \gamma_{t+1}}{n}} \! \sum_{i=1}^n \! \Tgrd f_i^t\rangle  +   {\textstyle \frac{\gamma_{t+1}^2 L}{2 n^2} }\left\| \sum_{i=1}^n \Tgrd f_i^t \right\|^2.
\end{align*}}Taking the conditional expectation $\EE_t[\cdot]$ on both sides yields
\begin{align} \label{eq:expF}
 \!\EE_t[ f( \Bprm^{t+1} ) ] &\leq f( \Bprm^t ) \!-\! \gamma_{t+1} \textstyle \Pscal{ \grd f( \Bprm^t ) }{ \frac{1}{n} \sum_{i=1}^n \grd f_i^t }\\
&\textstyle \quad + \gamma_{t+1}^2 \frac{L}{2} \EE_t \left[ \left\| \frac{1}{n} \sum_{i=1}^n \Tgrd f_i^t \right\|^2 \right]. \nonumber
\end{align}
We note that the second term is lower bounded by
\begin{align*}
&\textstyle \pscal{ \grd f( \Bprm^t ) }{ \frac{1}{n} \sum_{i=1}^n \grd f_i^t } \\
&\geq \textstyle \frac{1}{2} \| \grd f( \Bprm^t ) \|^2- \frac{1}{2} \left\| \frac{1}{n} \sum_{i=1}^n \big( \grd f_i( \Bprm^t) - \grd f_i^t \big) \right\|^2 ,
\end{align*}
and the last term can be upper bounded by
\begin{align*}
& \textstyle \EE_t \left[ \left\| \frac{1}{n} \sum_{i=1}^n \Tgrd f_i^t \right\|^2 \right] 
\leq \frac{\sigma^2}{n} + \left\| \frac{1}{n} \sum_{i=1}^n \grd f_i^t \right\|^2 \\
& \textstyle \leq \frac{\sigma^2}{n} + 2 \| \grd f( \Bprm^t) \|^2 + 2 \left\| \frac{1}{n} \sum_{i=1}^n \big( \grd f_i( \Bprm^t) - \grd f_i^t \big) \right\|^2,
\end{align*}
where we used A\ref{ass:SecOrdMom} in the first inequality. Substituting into \eqref{eq:expF} and using the step size condition $\sup_{t\geq 1}\gamma_{t} \leq \frac{1}{4L}$ gives
\beq \label{eq:expF1}
\EE_t[ f( \Bprm^{t+1} ) ] \leq f( \Bprm^t ) - \frac{\gamma_{t+1}}{4} \| \grd f( \Bprm^t ) \|^2 + \frac{ \gamma_{t+1}^2 L \sigma^2}{2n}\\
\textstyle + {2} \gamma_{t+1} \left\| \frac{1}{n} \sum_{i=1}^n \big( \grd f_i( \Bprm^t) - \grd f_i( \prm_i^t) \big) \right\|^2. \nonumber
\eeq
Observe that by A\ref{ass:smooth}-\ref{item:lips}, the last term is bounded by:
\begin{align*}
& \textstyle \left\| \frac{1}{n} \sum_{i=1}^n \big( \grd f_i( \Bprm^t) - \grd f_i^t \big) \right\|^2 
\leq \frac{L^2}{n} \norm{\BPrm^t-\Prm^t}^2_F.
\end{align*}
Substituting back into \eqref{eq:expF1} leads to \eqref{eq:concl-aa}.
\end{proof}
To prove Corollary~\ref{cor:1}, we construct the Lyapunov function:
\beqq
{\tt V}^{t} := \EE \left[ f( \Bprm^{t} ) - f^\star + \gamma_{t}\norm{\CSE{t}}^2_F \frac{{4}L^2}{\rho n} \right], ~~\forall t\geq 0.
\eeqq 
Combining Lemma \ref{lem:aa} and Lemma \ref{lem:ceb-ncvx}, we can get
\begin{align*}
\textstyle \sum_{t=0}^{T-1} \gamma_{t+1} \norm{\grd f(\Bprm^t)}^2 & \textstyle \leq 4 {\tt V}^{0} + \frac{2L \sigma^2}{n} \sum_{t=0}^{T-1} \gamma_{t+1}^{2}  \\
& \textstyle + \frac{{32} L^2 (\varsigma^2 + \sigma^2)}{\rho^2} \sum_{t=0}^{T-1} \gamma_{t+1}^{3}.
\end{align*}
Denote ${\sf D} := f(\Bprm^0) - f^\star$ and set the stepsize as $\gamma_{t} = (1/\sqrt{T}) \sqrt{2 {\sf D} n / (L \sigma^2)}$. Let {\sf T} be an random variable chosen uniformly from $\{0, \dots, T-1\}$, then 
\begin{align*}
    \EE[\normtxt{\grd f(\Bprm^{\sf T})}^2] \leq 
    \sqrt{\frac{32 {\sf D} L \sigma^2 }{n T}} + \frac{{64} L {\sf D} (\varsigma^2 +\sigma^2)}{ (\sigma^2/n) \rho^2 T} + \frac{ 16 L^2 \| \CSE{0} \|_F^2 }{ \rho n T }.
\end{align*}
which implies the bound (\ref{eq:corNor}) in Corollary \ref{cor:1}.

\subsection{Proof of Lemma \ref{lem:des2-ncvx}}\label{app:ncvx2}
 
We develop Lemma \ref{lem:des2-ncvx} from \eqref{eq:expF1}. Recall that ${\cal M}_i (\prm^\prime; \prm) \eqdef \grd f_i(\prm^{\prime}) - f_i(\prm) - \grd^2 f_i(\prm)(\prm^\prime-\prm)$ and \eqref{eq-is}, 
\beqq
\begin{aligned}
    & \frac{1}{n} \sum_{i=1}^{n}\left(\grd f_i^t-\grd f_i(\Bprm^t)\right) \\
    & = \frac{1}{n} \sum_{i=1}^{n} \left( {\cal M}_i(\prm_i^t; \Bprm^t) + [\grd^2 f_i(\Bprm^t)-\grd^2 f(\Bprm^t)](\prm_i^t - \Bprm^t) \right).
\end{aligned}
\eeqq
Under A\ref{ass:Hsmooth}, we have $\| \sum_{i=1}^n {\cal M}_i(\prm_i^t; \Bprm^t) \| \leq \frac{L_H}{2} \| \CSE{t} \|_F^2 $. Applying the triangular inequality leads to
\small \beqq
\begin{aligned}
\quad \norm{ \frac{1}{n} \sum_{i=1}^{n}\left(\grd f_i^t-\grd f_i(\Bprm^t)\right)} \leq \frac{L_H}{2n} \norm{\CSE{t} }_F^2 + \frac{\varsigma_H}{n} \sum_{i=1}^{n} \norm{\prm_i^t - \Bprm^t}.
\end{aligned}
\eeqq
\normalsize Taking square on both sides,
\begin{align} \label{eq:high-order-bd-1st}
    &\norm{ \frac{1}{n} \sum_{i=1}^{n}\left(\grd f_i^t-\grd f_i(\Bprm^t)\right)}^2
    \leq \frac{L_H^2}{2n^2} \norm{\CSE{t} }_F^4 + \frac{2\varsigma_H^2}{n} \norm{\CSE{t}}_F^2.
\end{align}
Substituting back into \eqref{eq:expF1} and requiring $\sup_{t\geq 1}\gamma_{t} \leq 1/(4L)$ leads to inequality (\ref{eq:des}). This concludes the proof.

\section{Missing Proofs for Strongly Convex Case}

\subsection{Proof of Corollary \ref{cor:2}}\label{app:cvx1}

\begin{lemma}[Descent Lemma]\label{lem:bb}
Under A\ref{ass:smooth},\ref{ass:str_cvx},\ref{ass:SecOrdMom}[$\sigma_0=\sigma_1=\sigma$], let the step sizes satisfy $\sup_{k\geq 1}\gamma_{k}\leq {\mu}/{( 8(\sigma^2+L^2) )}$, then we have the following inequality
\begin{eqnarray}\label{eq:ee}
   \begin{aligned}
       \EE_{t}\normtxt{\Tprm^{t+1}}^2 
       &\leq \left(1-\frac{\mu}{2}\gamma_{t+1}\right)\normtxt{\Tprm^{t}}^2 + \frac{2\sigma^2}{n}\gamma_{t+1}^2
       \\
       &\quad + \left( 4(\sigma^2+L^2) \gamma_{t+1}^2 + \frac{L^2}{ \mu}\gamma_{t+1}\right) \textstyle \frac{1}{n}\norm{\CSE{t}}_F^2.
   \end{aligned}
\end{eqnarray}
\end{lemma}

\begin{proof}
Using the recursion (\ref{eq:avg_it}), we have
\begin{align}\label{s:eq1}
    &\EE_{t} \normtxt{\Bprm^{t+1} - \prm^\star}^2 
    \\
    &= \EE_t \normtxt{\Bprm^{t} - \prm^\star}^2 + \gamma_{t+1}^2 \EE_t\bigg\| \frac{1}{n} \sum_{i=1}^{n}\Tgrd f_i^t \bigg\|^2 
        \nonumber\\
        &\quad - 2\gamma_{t+1}\Pscal{\Bprm^t-\prm^\star}{ \frac{1}{n} \sum_{i=1}^{n}\left(\grd f_i(\Bprm^t)-\grd f_i(\prm^\star)\right)}
        \nonumber\\
        &\quad - 2\gamma_{t+1}\Pscal{\Bprm^t-\prm^\star}{ \frac{1}{n} \sum_{i=1}^{n}\left(\grd f_i^t-\grd f_i(\Bprm^t)\right)}
    \nonumber\\
    &\overset{(a)}{\leq} (1-2\gamma_{t+1}\mu)\EE_t \norm{\Bprm^{t} - \prm^\star}^2 + \gamma_{t+1}^2 \EE_t \bigg\| \frac{1}{n} \sum_{i=1}^{n}\Tgrd f_i^t \bigg\|^2 
    \nonumber\\
        &\quad + 2\gamma_{t+1}\Pscal{\prm^\star-\Bprm^t}{ \frac{1}{n} \sum_{i=1}^{n}\left(\grd f_i^t-\grd f_i(\Bprm^t)\right)}
    \nonumber\\
    &\overset{(b)}{\leq} (1-\gamma_{t+1}\mu)\EE_t \norm{\Bprm^{t} - \prm^\star}^2 + \gamma_{t+1}^2 \EE_t \bigg\| \frac{1}{n} \sum_{i=1}^{n}\Tgrd f_i^t \bigg\|^2 
        \nonumber\\
        &\quad + \frac{\gamma_{t+1}}{\mu} \bigg\| \frac{1}{n} \sum_{i=1}^{n}\left(\grd f_i^t-\grd f_i(\Bprm^t)\right) \bigg\|^2,
        \nonumber
\end{align}
where (a) is due to the first-order optimality condition $\frac{1}{n}\sum_{i=1}^{n}\grd f_i(\prm^\star)=0$ and the $\mu$-strongly convexity of loss function, (b) is obtained by $a b \leq \frac{a^{2}}{2 \mu}+\frac{\mu b^{2}}{2}$. We consider:
\begin{align*}
    &\frac{\gamma_{t+1}}{\mu}\Big\|{ \frac{1}{n} \sum_{i=1}^{n}\left(\grd f_i^t-\grd f_i(\Bprm^t)\right)}\Big\|^2
    \\
    &\leq \frac{\gamma_{t+1}}{n \mu} \sum_{i=1}^{n} \norm{\grd f_i^t - \grd f_i(\Bprm^t)}^2
    \leq \frac{L^2}{n \mu} \gamma_{t+1} \norm{\CSE{t}}^2_F.
\end{align*}
Next, we bound the second term in the RHS of (\ref{s:eq1}).
\small \begin{align}
    &\EE_t \norm{ \frac{1}{n} \sum_{i=1}^{n}\Tgrd f_i^t}^2 
    \label{eq_ss}\\
    &\leq 2\EE_{t}\norm{\frac{1}{n}\sum_{i=1}^{n} ( \Tgrd f_i^t - \grd f_i^t) }^2 \!\!+\! 2\EE_{t}\norm{\frac{1}{n}\sum_{i=1}^{n} (\grd f_i^t - \grd f_i(\thstr))}^2
    \nonumber \\
    &\overset{(a)}{\leq} \frac{2}{n^2} \sum_{i=1}^{n} \EE_{t} \norm{\Tgrd f_i^t - \grd f_i^t}^2 + \frac{2}{n}\sum_{i=1}^{n} \EE_{t} \norm{\grd f_i^t - \grd f_i(\thstr)}^2
    \nonumber \\
    &\overset{(b)}{\leq} \frac{2}{n^2} \sum_{i=1}^{n} \sigma^2 (1+\norm{\prm_i^t - \thstr}^2) + \frac{2}{n} \sum_{i=1}^{n} L^2 \norm{\prm_i^t - \thstr}^2
    \nonumber \\
    &\leq \frac{2\sigma^2}{n} + \frac{4(\sigma^2 + L^2)}{n}\norm{\CSE{t}}_F^2 + 4(\sigma^2 + L^2)\norm{\Bprm^t - \thstr}^2,
    \nonumber
\end{align}
\normalsize where (a) is due to $\EE_t [\Tgrd f_i^t - \grd f_i^t] = 0$, (b) is due to A\ref{ass:smooth}, A\ref{ass:SecOrdMom}. 

Substituting the above upper bound to (\ref{s:eq1}) and using the notation $\Tprm^{t+1} \eqdef  \Bprm^{t+1}-\prm^\star$, we have
\beqq
\begin{aligned}
    & \EE_t\normtxt{\Tprm^{t+1}}^2 \leq  \left[1-\gamma_{t+1}\mu + 4(\sigma^2 
    + L^2)\gamma_{t+1}^2\right] \normtxt{\Tprm^t}^2
    \\
    &\quad 
    + \frac{2\sigma^2}{n}\gamma_{t+1}^{2}
    + \left(\frac{4(\sigma^2 + L^2)}{n}\gamma_{t+1}^2 + \frac{L^2}{n\mu} \gamma_{t+1} \right) \norm{\CSE{t}}_F^2 .
\end{aligned}
\eeqq
Setting $\sup_{k\geq 1}\gamma_{t+1}\leq \frac{\mu}{8(\sigma^2 + L^2)}$ concludes the proof.
\end{proof}

Our plan is to control $\EE \normtxt{\Tprm^t}^2$, $\EE \normtxt{ \CSE{t} }_F^2$ simultaneously. Consider the following Lyapunov function,
\beq\label{eq:u0}
\textstyle {\tt U}_{t+1} \eqdef \EE \left[ \normtxt{\widetilde{\prm}^{t+1}}^2 + \frac{4 L^2}{n \mu \rho}\gamma_{t+1}\norm{\CSE{t+1}}^2_F  \right], ~~ \forall t\geq 0.
\eeq
Combining the results from Lemma \ref{lem:bb} and Lemma \ref{lem:ceb-scvx}, if 
$\sup_{t\geq 1} \gamma_{t} \leq \min \big\{
\frac{\mu\rho}{8 \sigma L}, \frac{\rho}{2\mu}, \frac{L}{\sqrt{8\mu(\sigma^2+L^2)}}
\big\}$
then the following recursion holds for any $t\geq 1$,
\begin{align*}
    {\tt U}_{t+1} \leq (1-\frac{\mu}{4}\gamma_{t+1}) {\tt U}_{t}
    + \left( \frac{8(\sigma^2+\varsigma^2)L^2}{\mu\rho^2}\gamma_{t+1} +\frac{2\sigma^2}{n}  \right) \gamma_{t+1}^2
\end{align*}
Furthermore, suppose that the step sizes satisfy $\gamma_{t-1}/\gamma_{t} \leq \min\left\{\sqrt{1+\mu/4 \gamma_{t}^2}, \sqrt{1+\mu/4 \gamma_{t}^3} \right\}$, the recursion can be solved. Applying the auxiliary Lemma \ref{lem:aux} yields
\begin{align*}
    {\tt U}_{t+1} \leq \prod_{i=1}^{t+1}(1-\frac{\mu}{4}\gamma_{i}) {\tt U}_{0} + \frac{16\sigma^2}{n \mu}\gamma_{t+1} + \frac{64 L^2(\sigma^2 + \varsigma^2)}{\mu^2 \rho^2}\gamma_{t+1}^{2}.
\end{align*}
This immediately leads to 
\begin{align}\textstyle 
    \EE\normtxt{\widetilde{\prm}^{t+1}}^2 \leq  \prod_{i=1}^{t+1}(1 -\frac{\mu}{4}\gamma_{i}) {\tt U}_0 \!+\! \frac{16\sigma^2}{n\mu} \gamma_{t+1} + \frac{64(\sigma^2 + \varsigma^2) L^2 }{\mu^2 \rho^2} \gamma_{t+1}^2.
    \label{eq:standard-bound-scvx}
\end{align}
As the first term decreases sub-geometrically, we obtain Corollary \ref{cor:2} when $t$ is sufficiently large.

\subsection{Proof of Lemma \ref{lem:des2-scvx}}\label{app:cvx2}

We aim to derive a tighter bound than Lemma \ref{lem:bb}. Continuing the derivation from (\ref{s:eq1}) and applying \eqref{eq:high-order-bd-1st} yield:
\begin{align}
\EE_t \norm{\Tprm^{t+1}}^2 
&\leq (1-\mu\gamma_{t+1}) \norm{\Tprm^{t}}^2 +\gamma_{t+1}^2 \EE_t \norm{ \frac{1}{n} \sum_{i=1}^{n}\Tgrd f_i^t}^2  \notag
\\
&\quad + \frac{\gamma_{t+1}L_H^2}{2n^2 \mu }\norm{\CSE{t}}_F^4  + \frac{2\varsigma_H^2 \gamma_{t+1}}{ \mu n}\norm{\CSE{t}}_F^2 .\label{s:eq2}
\end{align}
Substituting (\ref{eq_ss}) into (\ref{s:eq2}) leads to
\begin{align*}
    &\quad \EE_t \norm{\Tprm^{t+1}}^2 
    \\
    & \leq \left[1 \!-\! \mu \gamma_{t+1} + 4(\sigma^2+L^2)\gamma_{t+1}^2\right] \norm{\Tprm^{t}}^2 +  \frac{\gamma_{t+1}L_H^2}{2n^2 \mu }\norm{\CSE{t}}_F^4
    \\
    &\quad + \norm{\CSE{t}}^2_F \left( \gamma_{t+1}^2\frac{4(\sigma^2+L^2)}{n} +  \frac{2\varsigma^2_H \gamma_{t+1}}{n\mu}\right) + \frac{2\sigma^2}{n} \gamma_{t+1}^2.
\end{align*}
Using the condition $\sup_{t\geq 1}\gamma_{t}\leq \frac{\mu}{8(\sigma^2+L^2)}$ completes the proof.

\subsection{Proof of Lemma \ref{lem:hp-bnd}} \label{app:hp-bnd}
We first apply Corollary~\ref{cor:2} to prove that
\beq\label{eq:dd}
    \EE\normtxt{\Tprm^{t+1}}^2 \leq \frac{32 \sigma^2}{n \mu}\gamma_{t+1} \eqdef {\rm C} \gamma_{t+1},
\eeq
for any $t \geq \max\{ t_0, t_1 \}$, where $t_0, t_1$ will be determined as follows. 
Notice that under the premise of \eqref{eq:standard-bound-scvx}, the desired \eqref{eq:dd} is implied by 
\begin{align*}
    \textstyle 
    \frac{8\sigma^2}{n\mu}\gamma_{t+1} \geq \frac{64(\sigma^2+\varsigma^2)L^2}{\mu^2\rho^2}\gamma_{t+1}^2,
    ~~
    \frac{8\sigma^2}{n\mu}\gamma_{t+1} \geq \prod_{i=1}^{t+1}(1-\frac{\mu}{4}\gamma_{i}) {\tt U}_{0}.
\end{align*}
Substituting $\gamma_{t}=a_0 /(a_1 + t)$ into the first inequality leads to the requirement that
\beq\label{eq:t1}
t \geq t_{1} := \frac{8 a_{0}nL^2(1+\varsigma^2/\sigma^2)}{\mu\rho^2}-a_{1}.
\eeq
We then solve the second inequality which can be satisfied by
\[
    \frac{8\sigma^2}{n\mu}\gamma_{t+1} \geq \exp\{-\frac{\mu}{4}\sum_{i=1}^{t+1}\gamma_{i}\}{\tt U}_{0} \geq \prod_{i=1}^{t+1}(1-\frac{\mu}{4}\gamma_{i}) {\tt U}_{0}.
\]
We define 
\beq\label{eq:t0}
    t_0 \eqdef \inf \Big\{ t\geq0 \mid \frac{8\sigma^2}{n\mu}\gamma_{t+1} \geq {\tt U}_{0}\exp\{-\frac{\mu}{4}(t+1)\gamma_{t+1}\} \Big\},
\eeq
such that the second inequality holds for any $t \geq t_0$. Furthermore, as $\gamma_t$ is non-increasing, the $t_0$ defined in the above is finite. Finally, \eqref{eq:dd} holds for any $t \geq \max \{ t_0, t_1 \}$.

Next, we derive a high probability bound for $\normtxt{\Tprm^t}^2$. Our idea is to construct a non-negative sequence $\{\delta_{t}\}_{t\geq 1}$ such that
\beq \label{eq:whp_tdelta}
    \PP({\normtxt{\Tprm^{t}}^2\geq \delta_{t}^2})\leq \tdelta\gamma_{t}^2,
\eeq 
for any $t\geq 1$. Using the Markov inequality, we obtain
\begin{align*}
    \PP({\normtxt{\Tprm^{t}}^2\geq \delta_t^2}) &\leq \frac{\EE[\normtxt{\Tprm^{t}}^2]}{\delta_t^2} \leq  \frac{{\rm C} \gamma_{t}}{\delta_t^2},
\end{align*}
where the last inequality is due to (\ref{eq:dd}). Setting $\delta_t^2 = {{\rm C}}/{(\tdelta\gamma_{t})}$ gives \eqref{eq:whp_tdelta}. 
Subsequently, the union bound shows that
\beqq 
\begin{aligned}
    \PP\left(\cap_{i=1}^{t+1}\left\{\normtxt{\Tprm^{i}}^2 \leq \delta_i^2\right\}\right)
    & = 1-\PP\left(\cup_{i=1}^{t+1}\left\{\normtxt{\Tprm^{i}}^2\geq \delta_i^2 \right\} \right)
    \\
    & \geq 1- \tdelta\sum_{i=1}^{t+1}\gamma_{i}^2 \geq 1-\tdelta\frac{a_0^2}{a_1} ,
\end{aligned}
\eeqq
where the last inequality due to the chain
\begin{align*}
    \sum_{i=1}^{t+1}\gamma_{i}^2 &\leq \sum_{i=1}^{+\infty}\gamma_{i}^2 = \sum_{i=1}^{+\infty}\frac{a_0^2}{(a_1+t)^2}
    \leq \int_{\RR_+} \frac{a_0^2}{(a_1+t)^2}{\rm d} t = \frac{a_0^2}{a_1}.
\end{align*}

\subsection{Proof of Theorem \ref{thm2}}\label{app:thm2}
We provide further details to the derivation of \eqref{eq:lt} and the proof of Theorem~\ref{thm2}. Combining Lemma  \ref{lem:ceb-scvx}, \ref{lem:hp-bnd} shows that
\small \begin{align}
    &\quad \EE_t \norm{\CSE{t+1}}^4_F \label{eq_lemres}
    \\
    &\leq (1-\frac{\rho}{2})\norm{\CSE{t}}_F^4 
     + 54 n^2  \left[\bar{\sigma}^4 + 4\varsigma^4 +8\bar{\sigma}^4\normtxt{\Tprm^t}^4\right]\frac{\gamma_{t+1}^4}{\rho^3}
    \nonumber\\
    &\leq (1-\frac{\rho}{2})\norm{\CSE{t}}_F^4 
     + 54 n^2  \left[\bar{\sigma}^4 + 4\varsigma^4 +8\bar{\sigma}^4\normtxt{\Tprm^t}^2\cdot \frac{\rm C}{\tdelta\gamma_{t}} \right]\frac{\gamma_{t+1}^4}{\rho^3}
     \nonumber
    \\
    &\leq (1-\frac{\rho}{2})\normtxt{\CSE{t}}_F^4 \!+\! \frac{54 n^2(\bar{\sigma}^4+4\varsigma^4)}{\rho^3}\gamma_{t+1}^4 
    \!+\! \frac{432 n^2\bar{\sigma}^4 {\rm C}}{\tdelta \rho^3}\gamma_{t+1}^3 \normtxt{\Tprm^t}^2,
    \nonumber
\end{align}
\normalsize holds with probability at least $1-\frac{\tdelta a_0^2}{a_1}$. 

Recall the definition of ${\tt L}_t$ in \eqref{eq:lyap-fct-main}. Combining Lemma \ref{lem:des2-scvx}, \ref{lem:ceb-scvx}, (\ref{eq_lemres}), the step size condition $\sup_{k\geq 1}\gamma_{k} \leq \min\big\{ \sqrt[3]{\frac{\mu}{8 c_2}} ,\frac{\rho}{\mu}\big\}$, shows that the following holds
\begin{align}
    {\tt L}_{t+1} &\leq (1-\frac{\mu}{4}\gamma_{t+1}){\tt L}_{t} + \frac{2\sigma^2}{n}\gamma_{t+1}^2 \label{eq-c}
    \\
    &\quad + {\rm D}\varsigma_H^2 \gamma_{t+1}^3 + {\rm E}\gamma_{t+1}^4
    + {\rm F}\gamma_{t+1}^5,~\forall~t \geq \max\{t_1 , t_2 \}. \notag
\end{align}
with probability at least $1-\tdelta a_0^2/a_1$.

With $\textstyle \frac{\gamma_{t+1}}{\gamma_{t}} \leq \min_{p \in \{2,3,4,5\}} \sqrt{1+\mu/4\gamma_{t}^{p}}$, 
solving the recursion (\ref{eq-c}) with Lemma \ref{lem:aux} gives 
\begin{align*} 
    {\tt L}_{t+1} & \leq \prod_{i=1}^{t+1}(1-\frac{\mu}{4}\gamma_{i}) {\tt L}_{0} + \frac{16 \sigma^2}{n\mu}\gamma_{t+1} + \frac{8 {\rm D}\varsigma_H^2}{\mu}\gamma_{t+1}^2 \notag
    \\
    &\quad + \frac{8 \rm E}{\mu}\gamma_{t+1}^3  + \frac{8 \rm F}{\mu}\gamma_{t+1}^4.
\end{align*}
The above inequality of ${\tt L}_{t+1}$ immediately leads to (\ref{eq_thm2_mse}) of Theorem \ref{thm2} as ${\tt L}_{t+1}$ is lower bounded by $\EE \normtxt{ \Tprm^{t+1} }^2 $.


\section{Analysis for Decentralized TD(0) Algorithm}

\subsection{Proof of Lemma \ref{lem:td_des}} \label{app:td0}
Recall that $\Tprm^{t+1} \eqdef \Bprm^{t+1}-\prm^\star$, we observe that,
\begin{align*}
    \norm{\Tprm^{t+1}}^2 = \norm{\Bprm^{t}-\prm^\star + \alpha_{t+1} \left[ \bar{\bm b}(\zeta^{t+1}) - {\bm A}(\zeta^{t+1})\Bprm^t \right]}^2.
\end{align*}
Taking the conditional expectation gives
\begin{align}
&\EE_{t} \norm{\Bprm^{t+1}-\prm^\star}^2 \notag
    \\
&=\norm{\Bprm^t-\prm^\star}^2 + \alpha_{t+1}^2 \EE_{t}\norm{\bar{\bm b}(\zeta^{t+1}) - {\bm A}(\zeta^{t+1})\Bprm^t}^2 \notag
    \\
&\quad - 2\alpha_{t+1} {\EE_{t} \Pscal{\Bprm^t-\prm^{\star}}{{\bm A}(\zeta^{t+1})\Bprm^t - \bar{\bm b}(\zeta^{t+1})}}.  \label{eq:split_quad}
\end{align}
Using ${\bm A} \prm^\star = \bar{\bm b}$, the last term in (\ref{eq:split_quad}) evaluates to
\begin{align*}
    & \Pscal{ \Bprm^t-\prm^{\star} } { {\bm A} \Bprm^t - \bar{\bm b} } = 
    \Pscal{ \Bprm^t-\prm^{\star} } { {\bm A} ( \Bprm^t - \prm^\star ) }
    \\
    & \geq {\lambdamin \norm{\Bprm^t - \prm^\star}^2},
\end{align*}
where $\lambda_{\text{min}}$ is the minimum eigenvalue of $\frac{{\bm A} + {\bm A}^\top}{2}$. For the second term in the RHS of \eqref{eq:split_quad}, we observe
\begin{align*} 
    &\EE_{t}\norm{\bar{\bm b}(\zeta^{t+1}) 
    - {\bm A}(\zeta^{t+1})\Bprm^t}^2 \\
    & = \EE_t \norm{ \bar{\bm b}( \zeta^{t+1} ) - \bar{\bm b} }^2 + \EE_t \norm{ {\bm A}(\zeta^{t+1}) \Bprm^t - {\bm A} \prm^\star }^2  \notag
    \\
    & \leq \frac{ \sigma^2 }{n} + \EE_t \norm{ ( {\bm A}(\zeta^{t+1}) - {\bm A} ) \prm^\star }^2 + \EE_t \norm{ {\bm A}(\zeta^{t+1}) ( \Bprm^t - \prm^\star) }^2  \notag
    \\
    & \leq \frac{ \sigma^2 }{n} + 4 \beta^2 \| \prm^\star \|^2 + \beta^2 \| \Bprm^t - \prm^\star \|^2 . \notag
\end{align*}
Substituting into \eqref{eq:split_quad} gives us
\begin{align*}
&\EE_{t} \| \Bprm^{t+1}-\prm^\star \|^2 \notag
    \\
& \leq \left( 1 - 2 \alpha_{t+1} \lambdamin + \alpha_{t+1}^2 \beta^2 \right) \| \Bprm^t-\prm^\star \|^2 \\
& + \alpha_{t+1}^2 ( { \sigma^2 }/{n} + 4 \beta^2 \| \prm^\star \|^2 ). \notag
\end{align*}
Observing that $\alpha_{t+1} \leq \lambdamin / \beta^2$ leads to $1 - 2 \alpha_{t+1} \lambdamin + \alpha_{t+1}^2 \beta^2 \leq 1 - \alpha_{t+1} \lambdamin$ yields the desired bound.


\section{Auxiliary Lemmas }

The following auxiliary lemma is quite standard, see \cite{qiangmulti} Appendix E for detailed proof.

\begin{lemma}\label{lem:aux}
    Let $a>0$, $p\in \ZZ_+$ and $\left\{\gamma_{k}\right\}_{k \geq 1}$ be a non-increasing sequence such that $\gamma_{1}<2 / a$. If $\gamma_{k-1}^p / \gamma_{k}^p \leq 1+(a / 2) \gamma_{k}^p$ for any $k \geq 1$, then for any $k \geq 2$,
\[\textstyle
    \sum_{j=1}^{k} \gamma_{j}^{p+1} \prod_{\ell=j+1}^{k}\left(1-\gamma_{\ell} a\right) \leq \frac{2}{a} \gamma_{k}^p.
\]
\end{lemma}

\bibliographystyle{ieeetr}
\bibliography{ecl.bib}

\begin{thebibliography}{10}

\bibitem{cohen2017projected}
K.~Cohen, A.~Nedi{\'c}, and R.~Srikant, ``On projected stochastic gradient
  descent algorithm with weighted averaging for least squares regression,''
  {\em IEEE Transactions on Automatic Control}, vol.~62, no.~11,
  pp.~5974--5981, 2017.

\bibitem{mateos2012distributed}
G.~Mateos and G.~B. Giannakis, ``Distributed recursive least-squares: Stability
  and performance analysis,'' {\em IEEE Transactions on Signal Processing},
  vol.~60, no.~7, pp.~3740--3754, 2012.

\bibitem{doan2019finite}
T.~Doan, S.~Maguluri, and J.~Romberg, ``Finite-time analysis of distributed td
  (0) with linear function approximation on multi-agent reinforcement
  learning,'' in {\em International Conference on Machine Learning},
  pp.~1626--1635, PMLR, 2019.

\bibitem{forrester2007multi}
A.~I. Forrester, A.~S{\'o}bester, and A.~J. Keane, ``Multi-fidelity
  optimization via surrogate modelling,'' {\em Proceedings of the royal society
  a: mathematical, physical and engineering sciences}, vol.~463, no.~2088,
  pp.~3251--3269, 2007.

\bibitem{nedic2017fast}
A.~Nedi{\'c}, A.~Olshevsky, and C.~A. Uribe, ``Fast convergence rates for
  distributed non-bayesian learning,'' {\em IEEE Transactions on Automatic
  Control}, vol.~62, no.~11, pp.~5538--5553, 2017.

\bibitem{cohen2016distributed}
K.~Cohen, A.~Nedi{\'c}, and R.~Srikant, ``Distributed learning algorithms for
  spectrum sharing in spatial random access wireless networks,'' {\em IEEE
  Transactions on Automatic Control}, vol.~62, no.~6, pp.~2854--2869, 2016.

\bibitem{brisimi2018federated}
T.~S. Brisimi, R.~Chen, T.~Mela, A.~Olshevsky, I.~C. Paschalidis, and W.~Shi,
  ``Federated learning of predictive models from federated electronic health
  records,'' {\em International journal of medical informatics}, vol.~112,
  pp.~59--67, 2018.

\bibitem{chang2020distributed}
T.-H. Chang, M.~Hong, H.-T. Wai, X.~Zhang, and S.~Lu, ``Distributed learning in
  the nonconvex world: From batch data to streaming and beyond,'' {\em IEEE
  Signal Processing Magazine}, vol.~37, no.~3, pp.~26--38, 2020.

\bibitem{nedic2020distributed}
A.~Nedic, ``Distributed gradient methods for convex machine learning problems
  in networks: Distributed optimization,'' {\em IEEE Signal Processing
  Magazine}, vol.~37, no.~3, pp.~92--101, 2020.

\bibitem{lian2017decentralized}
X.~Lian, C.~Zhang, H.~Zhang, C.-J. Hsieh, W.~Zhang, and J.~Liu, ``Can
  decentralized algorithms outperform centralized algorithms? a case study for
  decentralized parallel stochastic gradient descent,'' 2017.

\bibitem{pu2021sharp}
S.~Pu, A.~Olshevsky, and I.~C. Paschalidis, ``A sharp estimate on the transient
  time of distributed stochastic gradient descent,'' {\em IEEE Transactions on
  Automatic Control}, 2021.

\bibitem{C2}
A.~Koloskova, T.~Lin, and S.~U. Stich, ``An improved analysis of gradient
  tracking for decentralized machine learning,'' {\em Advances in Neural
  Information Processing Systems}, vol.~34, pp.~11422--11435, 2021.

\bibitem{pu2021distributed}
S.~Pu and A.~Nedi{\'c}, ``Distributed stochastic gradient tracking methods,''
  {\em Mathematical Programming}, vol.~187, no.~1, pp.~409--457, 2021.

\bibitem{xin2021improved}
R.~Xin, U.~A. Khan, and S.~Kar, ``An improved convergence analysis for
  decentralized online stochastic non-convex optimization,'' {\em IEEE
  Transactions on Signal Processing}, vol.~69, pp.~1842--1858, 2021.

\bibitem{C3}
S.~A. Alghunaim and K.~Yuan, ``A unified and refined convergence analysis for
  non-convex decentralized learning,'' {\em IEEE Transactions on Signal
  Processing}, vol.~70, pp.~3264--3279, 2022.

\bibitem{yuan2023removing}
K.~Yuan, S.~A. Alghunaim, and X.~Huang, ``Removing data heterogeneity influence
  enhances network topology dependence of decentralized sgd,'' {\em Journal of
  Machine Learning Research}, vol.~24, no.~280, pp.~1--53, 2023.

\bibitem{huang2022improving}
K.~Huang and S.~Pu, ``Improving the transient times for distributed stochastic
  gradient methods,'' {\em IEEE Transactions on Automatic Control}, 2022.

\bibitem{C4}
Y.~Lu and C.~De~Sa, ``Optimal complexity in decentralized training,'' in {\em
  International Conference on Machine Learning}, pp.~7111--7123, PMLR, 2021.

\bibitem{yuan2022revisiting}
K.~Yuan, X.~Huang, Y.~Chen, X.~Zhang, Y.~Zhang, and P.~Pan, ``Revisiting
  optimal convergence rate for smooth and non-convex stochastic decentralized
  optimization,'' {\em Advances in Neural Information Processing Systems},
  vol.~35, pp.~36382--36395, 2022.

\bibitem{tsitsiklis1984problems}
J.~N. Tsitsiklis, ``Problems in decentralized decision making and
  computation.,'' tech. rep., Massachusetts Inst of Tech Cambridge Lab for
  Information and Decision Systems, 1984.

\bibitem{ram2008distributed}
S.~S. Ram, A.~Nedich, and V.~V. Veeravalli, ``Distributed stochastic
  subgradient projection algorithms for convex optimization,'' {\em arXiv
  preprint arXiv:0811.2595}, 2008.

\bibitem{bianchi2012convergence}
P.~Bianchi and J.~Jakubowicz, ``Convergence of a multi-agent projected
  stochastic gradient algorithm for non-convex optimization,'' {\em IEEE
  transactions on automatic control}, vol.~58, no.~2, pp.~391--405, 2012.

\bibitem{nedic2017achieving}
A.~Nedic, A.~Olshevsky, and W.~Shi, ``Achieving geometric convergence for
  distributed optimization over time-varying graphs,'' {\em SIAM Journal on
  Optimization}, vol.~27, no.~4, pp.~2597--2633, 2017.

\bibitem{di2016next}
P.~Di~Lorenzo and G.~Scutari, ``Next: In-network nonconvex optimization,'' {\em
  IEEE Transactions on Signal and Information Processing over Networks},
  vol.~2, no.~2, pp.~120--136, 2016.

\bibitem{tang2018d}
H.~Tang, X.~Lian, M.~Yan, C.~Zhang, and J.~Liu, ``$ d^{2}$: Decentralized
  training over decentralized data,'' in {\em International Conference on
  Machine Learning}, pp.~4848--4856, PMLR, 2018.

\bibitem{yuan2018exact}
K.~Yuan, B.~Ying, X.~Zhao, and A.~H. Sayed, ``Exact diffusion for distributed
  optimization and learning—part i: Algorithm development,'' {\em IEEE
  Transactions on Signal Processing}, vol.~67, no.~3, pp.~708--723, 2018.

\bibitem{assran2019stochastic}
M.~Assran, N.~Loizou, N.~Ballas, and M.~Rabbat, ``Stochastic gradient push for
  distributed deep learning,'' in {\em International Conference on Machine
  Learning}, pp.~344--353, PMLR, 2019.

\bibitem{towfic2016excess}
Z.~J. Towfic, J.~Chen, and A.~H. Sayed, ``Excess-risk of distributed stochastic
  learners,'' {\em IEEE Transactions on Information Theory}, vol.~62, no.~10,
  pp.~5753--5785, 2016.

\bibitem{lian2017can}
X.~Lian, C.~Zhang, H.~Zhang, C.-J. Hsieh, W.~Zhang, and J.~Liu, ``Can
  decentralized algorithms outperform centralized algorithms? a case study for
  decentralized parallel stochastic gradient descent,'' {\em Advances in Neural
  Information Processing Systems}, vol.~30, 2017.

\bibitem{vogels2023beyond}
T.~Vogels, H.~Hendrikx, and M.~Jaggi, ``Beyond spectral gap: the role of the
  topology in decentralized learning,'' {\em Journal of Machine Learning
  Research}, vol.~24, no.~355, pp.~1--31, 2023.

\bibitem{ryabinin2021moshpit}
M.~Ryabinin, E.~Gorbunov, V.~Plokhotnyuk, and G.~Pekhimenko, ``Moshpit sgd:
  Communication-efficient decentralized training on heterogeneous unreliable
  devices,'' in {\em Advances in Neural Information Processing Systems},
  vol.~34, 2021.

\bibitem{hendrikx2020statistically}
H.~Hendrikx, L.~Xiao, S.~Bubeck, F.~Bach, and L.~Massoulie, ``Statistically
  preconditioned accelerated gradient method for distributed optimization,'' in
  {\em International conference on machine learning}, pp.~4203--4227, PMLR,
  2020.

\bibitem{li2022role}
Q.~Li and H.-T. Wai, ``On the role of data homogeneity in multi-agent
  non-convex stochastic optimization,'' in {\em 2022 IEEE 61st Conference on
  Decision and Control (CDC)}, pp.~5843--5848, IEEE, 2022.

\bibitem{beznosikov2021distributed}
A.~Beznosikov, G.~Scutari, A.~Rogozin, and A.~Gasnikov, ``Distributed
  saddle-point problems under data similarity,'' {\em Advances in Neural
  Information Processing Systems}, vol.~34, pp.~8172--8184, 2021.

\bibitem{sun2022distributed}
Y.~Sun, G.~Scutari, and A.~Daneshmand, ``Distributed optimization based on
  gradient tracking revisited: Enhancing convergence rate via surrogation,''
  {\em SIAM Journal on Optimization}, vol.~32, no.~2, pp.~354--385, 2022.

\bibitem{tian2022acceleration}
Y.~Tian, G.~Scutari, T.~Cao, and A.~Gasnikov, ``Acceleration in distributed
  optimization under similarity,'' in {\em International Conference on
  Artificial Intelligence and Statistics}, pp.~5721--5756, PMLR, 2022.

\bibitem{aldous1995reversible}
D.~Aldous and J.~Fill, {\em Reversible Markov chains and random walks on
  graphs}.
\newblock Berkeley, 1995.

\bibitem{ghadimi2013stochastic}
S.~Ghadimi and G.~Lan, ``Stochastic first-and zeroth-order methods for
  nonconvex stochastic programming,'' {\em SIAM Journal on Optimization},
  vol.~23, no.~4, pp.~2341--2368, 2013.

\bibitem{gower2019sgd}
R.~M. Gower, N.~Loizou, X.~Qian, A.~Sailanbayev, E.~Shulgin, and
  P.~Richt{\'a}rik, ``Sgd: General analysis and improved rates,'' in {\em
  International conference on machine learning}, pp.~5200--5209, PMLR, 2019.

\bibitem{arjevani2022lower}
Y.~Arjevani, Y.~Carmon, J.~C. Duchi, D.~J. Foster, N.~Srebro, and B.~Woodworth,
  ``Lower bounds for non-convex stochastic optimization,'' {\em Mathematical
  Programming}, pp.~1--50, 2022.

\bibitem{agarwal2012information}
A.~Agarwal, P.~L. Bartlett, P.~Ravikumar, and M.~J. Wainwright,
  ``Information-theoretic lower bounds on the oracle complexity of stochastic
  convex optimization,'' {\em IEEE Transactions on Information Theory},
  vol.~58, no.~5, pp.~3235--3249, 2012.

\bibitem{network_indenp}
S.~Pu, A.~Olshevsky, and I.~C. Paschalidis, ``Asymptotic network independence
  in distributed stochastic optimization for machine learning: Examining
  distributed and centralized stochastic gradient descent,'' {\em IEEE Signal
  Processing Magazine}, vol.~37, no.~3, pp.~114--122, 2020.

\bibitem{nedic2009distributed}
A.~Nedic and A.~Ozdaglar, ``Distributed subgradient methods for multi-agent
  optimization,'' {\em IEEE Transactions on Automatic Control}, vol.~54, no.~1,
  pp.~48--61, 2009.

\bibitem{wai2020accelerating}
H.-T. Wai, W.~Shi, C.~A. Uribe, A.~Nedi{\'c}, and A.~Scaglione, ``Accelerating
  incremental gradient optimization with curvature information,'' {\em
  Computational Optimization and Applications}, vol.~76, no.~2, pp.~347--380,
  2020.

\bibitem{bottou2018optimization}
L.~Bottou, F.~E. Curtis, and J.~Nocedal, ``Optimization methods for large-scale
  machine learning,'' {\em Siam Review}, vol.~60, no.~2, pp.~223--311, 2018.

\bibitem{nesterov2003introductory}
Y.~Nesterov, {\em Introductory lectures on convex optimization: A basic
  course}, vol.~87.
\newblock Springer Science \& Business Media, 2003.

\bibitem{sutton1988learning}
R.~S. Sutton, ``Learning to predict by the methods of temporal differences,''
  {\em Machine learning}, vol.~3, pp.~9--44, 1988.

\bibitem{di2022analysis}
S.~Di-Castro, S.~Mannor, and D.~Di~Castro, ``Analysis of stochastic processes
  through replay buffers,'' in {\em International Conference on Machine
  Learning}, pp.~5039--5060, PMLR, 2022.

\bibitem{liu2021distributed}
R.~Liu and A.~Olshevsky, ``Distributed td (0) with almost no communication,''
  {\em IEEE Control Systems Letters}, 2023.

\bibitem{bhandari2018finite}
J.~Bhandari, D.~Russo, and R.~Singal, ``A finite time analysis of temporal
  difference learning with linear function approximation,'' in {\em Conference
  on learning theory}, pp.~1691--1692, PMLR, 2018.

\bibitem{dalal2018finite}
G.~Dalal, B.~Sz{\"o}r{\'e}nyi, G.~Thoppe, and S.~Mannor, ``Finite sample
  analyses for td (0) with function approximation,'' in {\em Proceedings of the
  AAAI Conference on Artificial Intelligence}, 2018.

\bibitem{srikant2019finite}
R.~Srikant and L.~Ying, ``Finite-time error bounds for linear stochastic
  approximation andtd learning,'' in {\em Conference on Learning Theory},
  pp.~2803--2830, PMLR, 2019.

\bibitem{cutkosky2020momentum}
A.~Cutkosky and H.~Mehta, ``Momentum improves normalized sgd,'' in {\em
  International conference on machine learning}, pp.~2260--2268, PMLR, 2020.

\bibitem{fang2019sharp}
C.~Fang, Z.~Lin, and T.~Zhang, ``Sharp analysis for nonconvex sgd escaping from
  saddle points,'' in {\em Conference on Learning Theory}, pp.~1192--1234,
  PMLR, 2019.

\bibitem{le2023refined}
B.~Le~Bars, A.~Bellet, M.~Tommasi, E.~Lavoie, and A.-M. Kermarrec, ``Refined
  convergence and topology learning for decentralized sgd with heterogeneous
  data,'' in {\em International Conference on Artificial Intelligence and
  Statistics}, pp.~1672--1702, PMLR, 2023.

\bibitem{qiangmulti}
Q.~Li, C.-Y. Yau, and H.~T. Wai, ``Multi-agent performative prediction with
  greedy deployment and consensus seeking agents,'' in {\em Advances in Neural
  Information Processing Systems}, 2022.

\end{thebibliography}

\begin{IEEEbiography}
[{\includegraphics[width=1in,height=1.25in,clip,keepaspectratio]{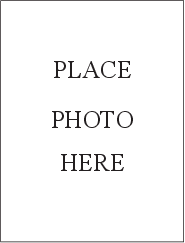}}]{Qiang Li} received the B.S. degree from the Harbin Institute of Technology, China in 2020. He is currently pursuing the Ph.D. degree in the System Engineering and Engineering Management at the Chinese University of Hong Kong (CUHK), Hong Kong, SAR. His research interests include distributed optimization, stochastic optimization.
    
\end{IEEEbiography}

\begin{IEEEbiography}[{\includegraphics[width=1in,height=1.25in,clip,keepaspectratio]{ieee/Place_Photo.eps}}]{Hoi-To Wai} received the Ph.D. degree in electrical engineering from Arizona State University (ASU), Tempe, AZ, USA, in Fall 2017, the B.Eng. (with First Class Hons.) and M.Phil. degrees in electronic engineering from The Chinese University of Hong Kong (CUHK), Hong Kong, in 2010 and 2012, respectively. He is currently an Assistant Professor with the Department of Systems Engineering and Engineering Management, CUHK. He has held research positions with ASU, University of California, Davis, CA, USA, Telecom ParisTech, Paris, France, Ecole Polytechnique, Palaiseau, France, LIDS, Massachusetts Institute of Technology, Cambridge, MA, USA. His research interests include signal processing, machine learning, and distributed optimization, with a focus of their applications to network science. His dissertation was the recipient of the 2017’s Dean’s Dissertation Award from the Ira A. Fulton Schools of Engineering of ASU, and Best Student Paper Award at ICASSP 2018. He is currently an Associate Editor for IEEE Transactions on Signal and Information Processing on Networks and Elsevier's Signal Processing.
\vspace{-1cm}
\end{IEEEbiography}

\end{document}